\documentclass[10pt,reqno]{amsart}
\usepackage{}
\usepackage{bbm}
\usepackage{amsfonts}
\usepackage{amssymb}
\usepackage{amsmath}
\usepackage{leftidx}
\usepackage{extarrows}
\usepackage[all]{xy}
\usepackage{cases}
\usepackage{txfonts}
\usepackage{amsfonts}
\usepackage{mathrsfs}
\usepackage{amssymb}
\usepackage{amsmath}
\usepackage{epic}
\usepackage[top=1in, bottom=1in, left=1.25in, right=1.25in]{geometry}
\newtheorem{proposition}{Proposition}[section]
\newtheorem{lemma}[proposition]{Lemma}

\newtheorem{theorem}[proposition]{Theorem}

\theoremstyle{definition}
\newtheorem{definition}[proposition]{Definition}
\newtheorem{example}[proposition]{Example}

\theoremstyle{remark}
\newtheorem{remark}[proposition]{Remark}

\newcommand{\lelabel}[1]{\label{le:#1}}
\newcommand{\leref}[1]{Lemma~\ref{le:#1}}
\newcommand{\prlabel}[1]{\label{pr:#1}}
\newcommand{\prref}[1]{Proposition~\ref{pr:#1}}

\newcommand{\relabel}[1]{\label{re:#1}}
\newcommand{\reref}[1]{Remark~\ref{re:#1}}
\newcommand{\exlabel}[1]{\label{ex:#1}}

\newcommand{\eqlabel}[1]{\label{eq:#1}}
\newcommand{\equref}[1]{(\ref{eq:#1})}

\def\a{\alpha}

\def\d{\delta}
\def\D{\Delta}

\def\ep{\varepsilon}
\def\End{\mathrm{End}}

\def\Hom{\mathrm{Hom}}

\def\Ker{\mathrm{Ker}}

\def\lhu{\leftharpoonup}
\def\rhu{\rightharpoonup}

\def\op{\oplus}
\def\ot{\otimes}

\def\ra{\rightarrow}

\def\ss{\subseteq}

\def\ti{\times}

\def\vf{\varphi}

\def\Z{\mathbb{Z}}

\def\<{\leq}
\def\>{\geq}

\date{}
\begin{document}
\title{Dorroh extensions of algebras and coalgebras, I}
\thanks{}
\author{Lan You}
\address{School of Mathematical Science, Yangzhou University, Yangzhou 225002, China;
School of Mathematics and Physics, Yancheng Institute
of Technology, Yancheng 224051, China}
\email{youl@ycit.cn}
\author{Hui-Xiang Chen}
\address{School of Mathematical Science, Yangzhou University, Yangzhou 225002, China}
\email{hxchen@yzu.edu.cn}

\begin{abstract}
In this article, we study Dorroh extensions of algebras and Dorroh extensions of coalgebras. 
Their structures are described. Some properties  of these extensions are presented.  
We also introduce the finite duals of algebras and modules which are not unital. 
Using these finite duals, we determine the dual relations between the two kinds of extensions.
\end{abstract}
\maketitle

\section*{Introduction}

Let $R$ and $S$ be two associative rings not necessarily containing identities. $S$ is said to be a
{\it Dorroh extension} of $R$ if $R$ is a subring of $S$ and there is an ideal $I$ of $S$
such that $S=R\op I$. Clearly, this is equivalent to that there is a projection from
$S$ onto $R$ as rings.

There is a general method to construct Dorroh extensions of rings as follows. Let $R$ and $I$ be
two associative rings not necessarily containing identities. Suppose that $I$ is an $R$-bimodule
such that the actions of $R$ on $I$ are compatible with the multiplication of $I$, that is,
$r(xy)=(rx)y$, $(xr)y=x(ry)$ and $(xy)r=x(yr)$ for all $r\in R$ and $x, y\in I$. Then $R\op I$ is an associative
ring with the multiplication given by $(r,x)(p,y)=(rp,xp+ry+xy)$, $r, p\in R$, $x, y\in I$.
Moreover, $R$ and $I$ cab be canonically regarded as additive subgroups of $R\op I$
via the identifying $r=(r,0)$ and $x=(0,x)$, respectively, where $r\in R$ and $x\in I$.
In this case, $R$ is a subring of $R\op I$ and $I$ is an ideal of $R\op I$.
Thus, $R\op I$ is a Dorroh extension of $R$, called a {\it Dorroh extension of $R$ by $I$}.
$(R, I)$ is called a {\it Dorroh pair of rings}. If $R$ has an identity $1_R$ and $RI=I=IR$,
then $(1_R, 0)$ is an identity of $R\op I$. When $R=\Z$, the ring of integers, this procedure is exactly
the method to embed  a ring (without identity) into a ring with identity given by Dorroh in \cite{D}.

In ring theory, Dorroh extension has become an important method for constructing new rings and investing properties of rings.
Many ring constructions  can be regarded as Dorroh extensions of rings,
for instance, the trivial extension of a ring with a bimodule, triangular matrix rings, $\mathbb{N}$-graded rings and so on. The
Dorroh extension $\Z \op I$ is used to describe the relative $K_0$-group
of a ring and an ideal in \cite[Section 1.5]{Ro}. Cibils etc. computed the Hochschild cohomology of Dorroh extension
(which is called {\it split algebra}) in \cite{CMRS}.
Furthermore, properties of Dorroh extensions of rings are referred to \cite{AJM, CS, DF, F, M}.

In this paper, we investigate Dorroh extensions of algebras, Dorroh extensions of coalgebras and the finite duals of algebras and their modules. This paper is organized as follows. In Section 1, we investigate Dorroh extensions of algebras and Dorroh pair of algebras. For any Dorroh pair $(A, I)$ of algebras, one can form an algebra $A\ltimes_d I$. $A\ltimes_d I$ is an algebra Dorroh extension of $A$ and  any algebra Dorroh extension of $A$ has such a form. It is shown that $A\ltimes_d I$ is isomorphic to the direct product algebra
$A\times I$ if $I$ is unital. A universal property for $A\ltimes_d I$ is given. 
Then some examples of Dorroh extension of algebra are given and the modules over $A\ltimes_d I$ are described.
 In Section 2, we introduce Dorroh extension of coalgebras and Dorroh pair of coalgebras. Some examples and properties are given. For any Dorroh pair $(C, P)$ of coalgebras, one can form a coalgebra $C\ltimes_d P$. It is shown that $C\ltimes_d P$ is a coalgebra Dorroh extension of $C$ and that any coalgebra Dorroh extension of $C$ is isomorphic to some $C\ltimes_d P$. 
It is shown that  $C\ltimes_d P$ is isomorphic to the direct product coalgebra $C\times P$
if $P$ is counital. A universal property for $C\ltimes_d P$ is given.
The comodules over $C\ltimes_d P$ are also described.
In Section 3, we first introduce the finite dual $A^{\circ}$ of an algebra $A$,
which is different from the usual definition (\cite[Lemma 9.1.1]{Mo}.
It is shown that the finite dual of an algebra is a coalgebra. Then we introduce the finite dual $M_r^{\circ}$
(resp., $M_l^{\circ}$) of a right (resp., left) module $M$ over an algebra $A$.
It is shown that $M_r^{\circ}$ (resp., $M_l^{\circ}$) is a right (resp., left) $A^{\circ}$-comudule,
and that if $M$ is an $A$-bimodule then $M^{\circ}_{(A)}=M_r^{\circ}\cap M_l^{\circ}$ is
an $A^{\circ}$-bicomudule. Then we describe the  the dual relations between the Dorroh extensions of algebras and the Dorroh extensions of coalgebras.
It is shown that if $(C, P)$ is a  Dorroh pair of coalgebras then $(C^*, P^*)$ is a Dorroh pair of algebras
and $(C\ltimes_d P)^*\cong C^*\ltimes_d P^*$ as algebras,
and that if $(A, I)$ is a  Dorroh pair of algebras then $(A^{\circ}, I^d)$ is a Dorroh pair of coalgebras
and $(A\ltimes_d I)^{\circ}\cong A^{\circ}\ltimes_d I^d$ as coalgebras, where $I^d=I^{\circ}\cap I^{\circ}_{(A)}$.
The finite dual of a module over a unital algebra is also described.

In the subsequent paper \cite{YC2}, we will discuss Dorroh extensions of bialgebra and Hopf algebras, 
the ideals of Doroh extensions of algebras and the subcoalgebras of Dorroh extensions of coalgebras.

Throughout this paper, we work over a field $k$. All algebras are associative $k$-algebras not necessarily containing identities. So each ideal of an algebra is also a subalgebra. An algebra is {\it unital} if it has an identity. All modules, morphisms (of algebras and of modules) may not preserve identities unless otherwise mentioned. Similarly, all coalgebras are coassociative $k$-coalgebras not necessarily having counits. So each subcoalgebra is also a coideal. A coalgebra is called {\it counital} if it has a counit. All comodules, morphisms (of coalgebras and of comodules) are not required to preserve counits unless otherwise stated.

For basic facts about coalgebras and Hopf algebras, the readers can refer to the books \cite{Abe, Mo, Sw}.

\section{Dorroh extensions of algebras}
In this section, we shall discuss some properties of Dorroh extensions of algebras. 
All the material presented here can be extended to the case of rings. In fact,
some results are known for the case of rings. 

Let $A$ and $B$ be two algebras. $B$ is called {\it an algebra Dorroh extension of $A$} if $A$ is a subalgebra of $B$ and there is an ideal $I$ of $B$ such that $B=A\oplus I$ as vestor spaces. In this case, $B$ is also called {\it an algebra Dorroh extension of $A$ by $I$}.
 
Let $A$ and $I$ be  two algebras. $(A, I)$ is called {\it a Dorroh pair of algebras} if $I$ is an $A$-bimodule such that the module actions
are compatible with the multiplication of $I$, that is,
\begin{equation}\eqlabel{d1}
   a(xy)=(ax)y, \quad  (xa)y=x(ay), \quad   (xy)a=x(ya),
\end{equation}
for all $a\in A$ and $x,y\in I$.
In this case, $I$ is called an {\it $A$-Algebra} \cite{CS}.
$I'$ is called  {\it an $A$-subalgebra} of $I$ if $I'$ is a subalgebra and a subbimodule of $I$. $J$ is called {\it an $A$-Ideal} of $I$ if $J$ is an ideal and a subbimodule of $I$.

Given a Dorroh pair $(A, I)$ of algebras, one can construct an algebra $A\ltimes_d I$ as follows: $A\ltimes_d I=A\op I$ as a vector space, the multiplication is defined by
\begin{equation}\eqlabel{d2}
  (a,x)(b,y)=(ab, ay+xb+xy),\ (a,x), (b, y)\in A\ltimes_dI.
\end{equation}
$A$ and $I$ can be canonically embedded into  $A\ltimes_d I$ via $A\hookrightarrow A\ltimes_d I$, $a\mapsto(a,0)$
and  $I\hookrightarrow A\ltimes_d I$, $i\mapsto(0, i)$, respectively. In this case, $A$ is a subalgebra of $A\ltimes_d I$
and $I$ is an ideal of $A\ltimes_d I$. Hence $A\ltimes_d I$ is an algebra Dorroh extension of $A$ by $I$. Conversely, if $B$ is an algebra Dorroh extension of $A$ by $I$, then $(A,I)$ is a Dorroh pair of algebras and $B\cong A\ltimes_dI$ as algebras, where the module actions of $A$ on $I$ are given by the multiplication.

Let $(A, I)$ be a Dorroh pair of algebras. If $A$ is unital and the actions of $A$ on $I$ are unital, i.e., $AI=I=IA$, then $(A, I)$ is called {\it unital}. In this case, $A\ltimes_d I$ is a unital algebra with the identity $(1_A,0)$. When $I$ is unital, the structure of $A\ltimes_d I$ is much simpler.

\begin{proposition}\prlabel{prop1.1}
Let $(A,I)$ be a Dorroh pair of algebras. If $I$ is a unital algebra with identity $1_I$, then $a1_I=1_Ia$ for any $a\in A$, and $A\ltimes_d I\cong A\times I$ as algebras.
\end{proposition}
\begin{proof}
Suppose that $(A,I)$ is a Dorroh pair of algebras. By Eqs.\equref{d1}, $a1_I=1_I(a1_I)=(1_Ia)1_I=1_Ia$ for all $a\in A$.
Define a map $\eta: A\ltimes_d I\ra A\ti I$ by $\eta(a,x)=(a, x+a1_I)$. Clearly, $\eta$ is a linear isomorphism.
For any $(a,x), (b,y)\in A\ltimes_d I$,
we have $ \eta((a,x)(b,y))=\eta(ab,ay+xb+xy)=(ab,ay+xb+xy+ab1_I)$ and
\begin{equation*}
\begin{split}
    \eta(a,x)\eta(b,y)&=(a,x+a1_I)(b,y+b1_I)\\
    &=(ab,xy+(a1_I)y+x(b1_I)+(a1_I)(b1_I))\\
    &=(ab,xy+a(1_Iy)+(x1_I)b+ab1_I)\\
    &=(ab,xy+ay+xb+ab1_I).
\end{split}
\end{equation*}
Hence $\eta$ is an algebra isomorphism.
\end{proof}

For Dorroh pairs $(A,I)$ and $(A',I')$, a pair $(\phi, f)$ is called a {\it Dorroh homomorphism} from $(A,I)$ to $(A',I')$
if $\phi: A\ra A'$ and $f:I\ra I'$ are algebra homomorphisms satisfying $f(ax)=\phi(a)f(x)$ and $f(xa)=f(x)\phi(a)$
for all $a\in A$ and $x\in I$.
Let $\mathcal{DPA}$ be the category of all Dorroh pairs of algebras and Dorroh homomorphisms.

Let $\tau_A: A\ra A\ltimes_d I$ and $\tau_I: I\ra A\ltimes_d I$ be the canonical embedding maps. Then $\tau_I(ax)=\tau_A(a)\tau_I(x)$ and $\tau_I(xa)=\tau_I(x)\tau_A(a)$, that is, $(\tau_A, \tau_I)$ is
a Dorroh homomorphism from $(A,I)$ to $(A\ltimes_d I, A\ltimes_d I)$.

The following result describes that Dorroh extension $A\ltimes_d I$ and the embedding maps
$\tau_A$ and $\tau_I$ have a universal property in the category of all algebras. For the case of rings, one can refer to \cite[Theorem 7]{F}.
\begin{proposition}
Let $(A,I)$ be a Dorroh pair of algebras and $B$ an algebra. If $(\vf, f): (A, I)\ra(B, B)$ is a Dorroh homomorphism, then there is a unique algebra homomorphism $\eta: A\ltimes_d I\ra B$ such that $\eta\tau_A=\vf$ and $\eta\tau_I=f$.
\end{proposition}
\begin{proof}
Define a map $\eta: A\ltimes_d I\ra B$ by $\eta(a,i)=\vf(a)+f(i)$, $(a,i)\in A\ltimes_d I$. Then $\eta$ is linear.
For any $(a,x),(b,y)\in A\ltimes_d I$, $\eta ((a,x)(b,y))=\eta(ab,ay+xb+xy)=\vf(ab)+f(ay+xb+xy)
=\vf(a)\vf(b)+\vf(a)f(y)+f(x)\vf(b)+f(x)f(y)=\eta(a,x)\eta(b,y)$.
Hence $\eta$ is an algebra homomorphism.
Clearly, $\vf=\eta\tau_A$ and $f=\eta\tau_I$. If $\eta': A\ltimes_d I\ra B$ is also an algebra homomorphism such that $\vf=\eta'\tau_A$ and $f=\eta'\tau_I$, $\eta'(a,x)=\eta'(\tau_A(a)+\tau_I(x))=\vf(a)+f(x)=\eta(a,x)$
for any $(a,x)\in A\ltimes_d I$. Therefore, $\eta$ is unique.
\end{proof}

Now we give some examples of Dorroh pairs and Dorroh extensions.

\begin{example}\exlabel{ex1.6}
(a) Let $I$ be an algebra without identity. Then $k\ltimes_d I$ is a unital algebra with the identity
$(1,0)$.

(b) Let $A$ and $B$ be two algebras. Regard $B$ as an $A$-bimodule with the trivial actions defined by $ab=ba=0$ for any $a\in A$ and
$b\in B$. Then $(A, B)$ is a Dorroh pair of algebra and $A\ltimes _d B$ is exactly the direct product algebra $A\times B$.

(c) Let $A$ be an algebra and $M$ an $A$-bimodule. Then one can form a trivial extension $A\ltimes M$ as follows: $A\ltimes M=A\op M$ as a vector space, and the multiplication is defined by $(a,m)(a',m')=(aa',am'+ma')$, $a,a'\in A$, $m,m'\in M$. If we set $mm'=0$ for any $m,m'\in M$,
then $(A,M)$ is a Dorroh pair and the Dorroh extension $A\ltimes_d M$ is exactly the trivial extension $A\ltimes M$.

(d) Let $A$ and $B$ be two algebras, $M$ an $A$-$B$-bimodule. Then one can construct the triangular
matrix algebra
$\left(\begin{array}{cc}
                     A & M \\
                     0 & B \\
                   \end{array}
\right)$.
Note that $M$ is also an $A\times B$-bimodule with the actions given by $(a, b)m:=am$ and $m(a, b):=mb$, $a\in A$, $b\in B$, $m\in M$.
In this case, the trivial extension $(A\times B)\ltimes M$ is isomorphic to the triangular matrix algebra
$\left(\begin{array}{cc}
                     A & M \\
                     0 & B \\
                   \end{array}
                 \right)$. In particular, one-point extension $\left(
                   \begin{array}{cc}
                     A & M \\
                     0 & k \\
                   \end{array}
                 \right)$ is a Dorroh extension.

 (e) Let $(A, I)$ be a Dorroh pair of algebras. Consider the multiplicative monoid $\mathbb{Z}_2=\{0,1\}$ with $0\cdot 0=0\cdot 1=1\cdot 0=0$ and $1\cdot 1=1$.  Then $B=A\ltimes_d I$ is a $\mathbb{Z}_2$-graded algebra with  $B_1:=\tau_A(A)=A$ and $B_0:=\tau_I(I)=I$.
Conversely, any $\mathbb{Z}_2$-graded algebra $A=A_1\op A_0$ is a Dorroh extension of $A_1$ by $A_0$
since $A_1$ is a subalgebra of $A$ and $A_0$ is an ideal of $A$.

(f) Let $B=\oplus_{i=0}^{\infty}B_i$ be an $\mathbb{N}$-graded algebra.
Set $I:= \oplus_{i=1}^{\infty}B_i$. Then $B$ is a Dorroh extension of $B_0$ by $I$.
\end{example}

Now we describe the modules over $A\ltimes_dI$.

\begin{proposition}\prlabel{prop1.7}
Let $(A,I)$ be a Dorroh pair of algebras and $M$ a vector space. Then $M$ is a left (resp., right) $A\ltimes_d I$-module
if and only if $M$ is both a left (resp., right) $A$-module and a left (resp., right) $I$-module such that
$$a(xm)=(ax)m,\quad x(am)=(xa)m, \quad (\text{resp}., \, (mx)a=m(xa),\quad(ma)x=m(ax)),$$
where $a\in A$, $x\in I$ and $m\in M$.

If this is the case, then $M$ is an $A\ltimes_d I$-bimodule
if and only if $M$ is an $A$-bimodule, an $I$-bimodule, an $A$-$I$-bimodule and an $I$-$A$-bimodule.
\end{proposition}

\begin{proof}
As stated before, we always identify $(a,0)$ with $a$, $(0,x)$ with $x$, where $a\in A$ and $x\in I$. 
Hence $(a,x)=a+x$ in $A\ltimes_d I$.

Suppose that $M$ is a left $A\ltimes_d I$-module. Since $A$ and $I$ are subalgebras of $A\ltimes_d I$, 
$M$ is both a left $A$-module and a left $I$-module. Moreover, for any $a\in A$, $x\in I$ and $m\in M$, we have
$a(xm)=(a, 0)((0, x)m)=((a, 0)(0, x))m=(0, ax)m=(ax)m$ and $x(am)=(0, x)((a, 0)m)=((0, x)(a, 0))m=(0, xa)m=(xa)m$.
Conversely, suppose that $M$ is both a left $A$-module and a left $I$-module such that $a(xm)=(ax)m$ and $x(am)=(xa)m$
for any $a\in A$, $x\in I$ and $m\in M$. Define a left action of $A\ltimes_d I$ on $M$ by
$(a, x)m:=am+xm)$, $a\in A$, $x\in I$, $m\in M$. Then for any $a, b\in A$, $x, y\in I$ and $m\in M$, we have
$(a,x)((b,y)m)=(a,x)(bm+ym)=a(bm+ym)+x(bm+ym)=a(bm)+a(ym)+x(bm)+x(ym)
=(ab)m+(ay)m+(xb)m+(xy)m=(ab)m+(ay+xb+xy)m=(ab,ay+xb+xy)m=((a,x)(b,y))m$. 
Hence $M$ is a left $A\ltimes_d I$-module. For the right module case, the proof  is similar.
By a similar argument, one can check the last statement of the proposition.
\end{proof}

%

\begin{proposition}\prlabel{}
Let $A_1,A_2,A_3$ be algebras and $(A_1,A_2)\in \mathcal{DPA}$. Then $(A_1\ltimes_d A_2, A_3)\in \mathcal{DPA} $ if and only if $(A_1, A_3),(A_2,A_3) \in \mathcal{DPA} $, $A_3$ is both an $A_1$-$A_2$-bimodule and an $A_2$-$A_1$-bimodule, and $a_1(a_2a_3)=(a_1a_2)a_3$, $a_2(a_1a_3)=(a_2a_1)a_3$, $(a_3a_2)a_1=a_3(a_2a_1)$, $(a_3a_1)a_2=a_3(a_1a_2)$, $\forall a_i\in A_i$, $i=1,2,3$. Moreover, if this is the case, then $(A_1, A_2 \ltimes_d A_3)\in \mathcal{DPA} $ and 
$(A_1\ltimes_dA_2)\ltimes_dA_3\cong A_1\ltimes_d(A_2\ltimes_dA_3)$ as algebras.
\end{proposition}
\begin{proof}
The first statement follows from a straightforward verification similar to \prref{prop1.7}.
For the second statement, one can easily check that the linear isomorphism
$$(A_1\ltimes_d A_2) \ltimes_d A_3\ra A_1\ltimes_d(A_2\ltimes_d A_3), \ ((a_1,a_2),a_3)\mapsto (a_1,(a_2,a_3))$$
is an algebra homomorphism.
\end{proof}

\begin{remark}
The algebras $(A_1\ltimes_dA_2)\ltimes_dA_3$ and  $A_1\ltimes_d(A_2\ltimes_dA_3)$ given in
\prref{prop1.7} are simply written as $A_1\ltimes_dA_2\ltimes_dA_3$. 
More generally, one can define Dorroh $n$-tuple of algebras
$(A_1, A_2, \cdots, A_n)$ and form an iterated Dorroh extension of algebra  
$A_1\ltimes_d A_2\ltimes_d \cdots \ltimes_d A_n$
\end{remark}

\section{Dorroh extensions of coalgebras}

In this section, we study Dorroh extensions of coalgebras.
For simplicity, we use $1$ denote the identity map on some spaces.
For a coalgebra $C$ and $c\in C$, we write $\D(c)=\sum c_1\ot c_2$.
If $(M, \rho)$ is a right (resp., left) $C$-comodule, we write $\rho(m)=\sum m_{(0)}\ot m_{(1)}$
(resp.,  $\rho(m)=\sum m_{(-1)}\ot m_{(0)}$), $m\in M$.

\begin{definition}
Let $C$ and $D$ be two coalgebras. $D$ is called {\it a coalgebra Dorroh extension of $C$}
if $C$ is a subcoalgebra  of $D$ and there is a coideal $P$ of $D$ such that $D=C\op P$ as vector spaces.
In this case, $D$ is also called {\it a coalgebra Dorroh extension of $C$ by $P$}.
\end{definition}

\begin{definition}
Let $C$ and $P$ be two coalgebras. $(C,P)$ is called {\it a Dorroh pair of coalgebras} if $P$ is a $C$-bicomodule
such that the comodule structures are compatible with the comultiplication of $P$, 
that is, the following three equations hold for all $p\in P$:
\begin{eqnarray}
   \eqlabel{e3}\sum p_1\ot p_{2(0)}\ot p_{2(1)}&=&\sum p_{(0)1}\ot p_{(0)2}\ot p_{(1)}, \\
  \eqlabel{e4}  \sum p_{1(-1)}\ot p_{1(0)}\ot p_2&=&\sum p_{(-1)}\ot p_{(0)1}\ot p_{(0)2}, \\
  \eqlabel{e5} \sum p_{1(0)}\ot p_{1(1)}\ot p_{2}&=&\sum p_1\ot p_{2(-1)}\ot p_{2(0)}.
\end{eqnarray}
In this case, $P$ is called {\it a $C$-Coalgebra}. 
 $P'$ is called {\it a $C$-subCoalgebra of $P$} if
$P'\ss P$ is a subcoalgebra and a subbicomodule. $Q$ is called {\it a $C$-Coideal of $P$} if $Q\ss P$
is a coideal and a subbicomodule.
\end{definition}

Let $(C,P)$ be a Dorroh pair of coalgebras. Then one can form a coalgebra $C\ltimes_dP$ as follows:
$C\ltimes_dP=C\op P$ as vector spaces, the comultiplication is defined by
\begin{equation}\eqlabel{e6}
\begin{split}
 \D(c,p)=&\sum (c_1,0)\ot (c_2,0)+\sum (p_{(-1)},0)\ot (0,p_{(0)})\\
 &+\sum (0,p_{(0)})\ot (p_{(1)},0)+\sum (0,p_1)\ot (0,p_2)
 \end{split}
\end{equation}
for any $(c,p)\in C\ltimes_d P$. Clearly, $\D$ is linear map from $C\ltimes_dP$
to $(C\ltimes_dP)\ot (C\ltimes_dP)$. We need to show the coassociativity of the comultiplication.

\begin{lemma}\lelabel{lem2.3}
The comultiplication $\D$ defined by \equref{e6} is coassociative.
\end{lemma}
\begin{proof}
Clearly, $({\rm id}\ot \D)\D(c, 0)=(\D\ot{\rm id})\D(c, 0)$ for any $c\in C$.
Now let $p\in P$. Then
$$\begin{array}{rl}
({\rm id}\ot \D)\D(0, p)
=&\sum (p_{(-1)},0)\ot\D(0,p_{(0)})
+\sum (0,p_{(0)})\ot \D(p_{(1)},0)+\sum (0,p_1)\ot \D(0,p_2)\\
=&\sum (p_{(-2)},0)\ot(p_{(-1)},0)\ot\D(0,p_{(0)})
+\sum (p_{(-1)},0)\ot (0,p_{(0)})\ot \D(p_{(1)},0)\\
&+\sum (p_{(-1)},0)\ot (0,p_{(0)1})\ot (0,p_{(0)2})
+\sum (0,p_{(0)})\ot (p_{(1)},0)\ot (p_{(2)},0)\\
&+\sum (0,p_1)\ot(p_{2(-1)},0)\ot(0,p_{2(0)})
+\sum (0,p_1)\ot(0,p_{2(0)})\ot (p_{2(1)},0)\\
&+\sum (0,p_1)\ot(0,p_{2})\ot (0,p_{3})\\
\end{array}$$
and
$$\begin{array}{rl}
(\D\ot{\rm id})\D(0, p)
=&\sum \D(p_{(-1)},0)\ot(0,p_{(0)})
+\sum \D(0,p_{(0)})\ot (p_{(1)},0)+\sum \D(0,p_1)\ot (0,p_2)\\
=&\sum (p_{(-2)},0)\ot(p_{(-1)},0)\ot\D(0,p_{(0)})
+\sum (p_{(-1)},0)\ot (0,p_{(0)})\ot (p_{(1)},0)\\
&+\sum (0,p_{(0)})\ot (p_{(1)},0)\ot (p_{(2)},0)
+\sum (0, p_{(0)1})\ot (0,p_{(0)2})\ot (p_1,0)\\
&+\sum (p_{1(-1)},0)\ot(0,p_{1(0)})\ot (0, p_2)
+\sum (0,p_{1(0)})\ot(p_{1(1)},0)\ot (0,p_{2})\\
&+\sum (0,p_1)\ot(0,p_{2},0)\ot(0,p_{3}).\\
\end{array}$$
Hence $({\rm id}\ot \D)\D(0, p)=(\D\ot{\rm id})\D(0, p)$ by Eqs.\equref{e3}-\equref{e5}.
Thus, $({\rm id}\ot \D)\D(c, p)=(\D\ot{\rm id})\D(c, p)$ for any $(c,p)\in C\ltimes_dP$.
\end{proof}

$C$ and $P$ can be canonically embedded into $C\ltimes_d P$ as subspaces via
$\tau_C: C\hookrightarrow C\ltimes_dP$, $c\mapsto (c,0)$ and
$\tau_P: P\hookrightarrow C\ltimes_dP$, $p\mapsto (0,p)$. 
In this case, we have the following proposition.

\begin{proposition}\prlabel{}
Let $(C, P)$ be a Dorroh pair of coalgebras. Then  $C\ltimes_d P$ is a coalgebra Dorroh extension of $C$ by $P$.
Furthermore, If $C$ is a counital coalgebra with counit $\ep_C$ and $P$ is a counital $C$-bicomodule,
i.e., $\sum\ep_C(p_{(-1)})p_{(0)}=p=\sum p_{(0)}\ep_C(p_{(1)})$ for all $p\in P$,
then $C\ltimes_d P$ is also a counital coalgebra with counit $(\ep_C, 0)$.
\end{proposition}

\begin{proof}
By 	\leref{lem2.3}, $C\ltimes_d P$ is a coalgebra. Via the canonical embeddings $\tau_C$ and $\tau_P$ given above,
$C$ and $P$ are subspaces of $C\ltimes_d P$ and $C\ltimes_d P$ is equal to the direct sum of the two subspaces. 
By Eq.\equref{e6}, it is easy to see that $C$ is a subcoalgebra of $C\ltimes_d P$
and $P$ is a coideal of $C\ltimes_d P$. Hence  $C\ltimes_d P$ is a coalgebra Dorroh extension of $C$ by $P$.
The second statement follows from a straightforward verification.
\end{proof}

Note that the comultiplication $\D_P$ of $P$ is different from $\D_{C\ltimes_d P}|_P$, the restriction on $P$ of the
comultiplication of $C\ltimes_d P$. If it doesn't make confusion, we will regard $C\subseteq  C\ltimes_d P$
and $P\subseteq C\ltimes_d P$.

Let $D$ be a coalgebra Dorroh extension of $C$ by $P$. Then $C$ is a subcoalgebra  of $D$, 
$P$ is a coideal of $D$ and $D=C\op P$ as vector spaces. 
Let $\pi_C: D\ra C$ and $\pi_P: D\ra P$ be the corresponding linear projections,
and let $\tau_C: C\hookrightarrow D$ and $\tau_P: P\hookrightarrow D$ be the inclusion maps.
Then $\pi_C$ and $\tau_C$ are coalgebra homomorphisms.
Define linear maps $\D_P=(\pi_P\ot\pi_P)\D\tau_P: P\ra P\ot P$, $\rho_l=(\pi_C\ot\pi_P)\D\tau_P: P\ra C\ot P$ 
and $\rho_r=(\pi_P\ot\pi_C)\D\tau_P: P\ra P\ot C$, where $\D$ is the comultiplication of $D$.
Then we have the following proposition.

\begin{proposition}\prlabel{2.5}
Let $D$ be a coalgebra Dorroh extension of $C$ by $P$. Then $(C, P)$ is a Dorroh pair of coalgebras
and $D\cong C\ltimes_dP$ as coalgebras, where the comultiplication of $P$, 
the left and right $C$-comodule structure maps of $P$ are $\D_P$, 
$\rho_l$ and $\rho_r$ given above, respectively. 
\end{proposition} 

\begin{proof}
Since $C$ is a subcoalgebra of $D$, $(\pi_P\ot\pi_P)\D\tau_C=0$. Hence
$(\pi_P\ot\pi_P)\D\tau_P\pi_P=(\pi_P\ot\pi_P)\D(\tau_C\pi_C+\tau_P\pi_P)=(\pi_P\ot\pi_P)\D$.
Thus, we have 
$$\begin{array}{rl}
(\D_P\ot 1)\D_P=&((\pi_p\ot\pi_P)\D\tau_P\ot 1)(\pi_P\ot\pi_P)\D\tau_P
=((\pi_p\ot\pi_P)\D\tau_P\pi_P\ot \pi_P)\D\tau_P\\
=&((\pi_p\ot\pi_P)\D\ot \pi_P)\D\tau_P
=(\pi_p\ot\pi_P\ot\pi_P)(\D\ot 1)\D\tau_P,\\
\end{array}$$
and similarly $(1\ot\D_P)\D_P=(\pi_p\ot\pi_P\ot\pi_P)(1\ot\D)\D\tau_P$.
It follows that $(\D_P\ot 1)\D_P=(1\ot\D_P)\D_P$.
Similarly, one can check that $(1\ot\rho_l)\rho_l=(\D_C\ot 1)\rho_l$,
$(\rho_r\ot 1)\rho_r=(1\ot\D_C)\rho_r$, $(1\ot\rho_r)\rho_l=(\rho_l\ot 1)\rho_r$, 
$(\rho_l\ot 1)\D_P=(1\ot\D_P)\rho_l$, $(1\ot\rho_r)\D_P=(\D_P\ot 1)\rho_r$ and
$(\rho_r\ot 1)\D_P=(1\ot\rho_l)\D_P$. Therefore, $(C,P)$ is a Dorroh pair of coalgebras.

For any $x\in D$, we write $\D(x)=\sum x_1\ot x_2$. Then
$\D_P(p)=\sum\pi_P(p_1)\ot\pi_P(p_2)$, $\rho_l(p)=\sum\pi_C(p_1)\ot\pi_P(p_2)$
and $\rho_r(p)=\sum\pi_P(p_1)\ot\pi_C(p_2)$ for all $p\in P$.  Define a map
$\phi: C\ltimes_dP\ra D$ by $\phi(c, p)=c+p$, $(c,p)\in C\ltimes_dP$.
Clearly, $\phi$ is a linear isomorphism. 
Since $P$ is a coideal of $D$, $\D(P)\subseteq C\ot P+P\ot C$.
Hence $(\pi_C\ot\pi_C)\D(P)=0$ and so $\sum\pi_C(p_1)\ot\pi_C(p_2)=0$ for any $p\in P$.
Let $(c,p)\in C\ltimes_dP$.  Then
$$\begin{array}{rl}
\sum p_1\ot p_2=&\sum(\pi_C(p_1)+\pi_P(p_1))\ot(\pi_C(p_2)+\pi_P(p_2))\\
=&\sum\pi_C(p_1)\ot\pi_C(p_2)+\sum\pi_C(p_1)\ot\pi_P(p_2)\\
&+\sum \pi_P(p_1)\ot\pi_C(p_2)+\sum \pi_P(p_1)\ot \pi_P(p_2)\\
=&\sum\pi_C(p_1)\ot\pi_P(p_2)
+\sum \pi_P(p_1)\ot\pi_C(p_2)+\sum \pi_P(p_1)\ot \pi_P(p_2)\\
\end{array}$$
and hence
$$\begin{array}{rl}
(\phi\ot\phi)\D(c,p)=&(\phi\ot\phi)(\sum(c_1, 0)\ot(c_2,0)+\sum(\pi_C(p_1),0)\ot(0,\pi_P(p_2))\\
&+\sum (0, \pi_P(p_1))\ot(\pi_C(p_2), 0)+\sum(0, \pi_P(p_1))\ot(0, \pi_P(p_2)))\\
=&\sum c_1\ot c_2+\sum\pi_C(p_1)\ot\pi_P(p_2)\\
&+\sum \pi_P(p_1)\ot\pi_C(p_2)+\sum \pi_P(p_1)\ot \pi_P(p_2)\\
=&\sum c_1\ot c_2+\sum p_1\ot p_2\\
=&\D\phi(c,p).
\end{array}$$
It follows that $\phi$ is a coalgebra homomorphism, and so 
$D\cong C\ltimes_dP$ as coalgebras.
\end{proof}

Now we investigate the structure of $C\ltimes_d P$ in case $P$ is a counital coalgebra.
The following lemma is obvious.

\begin{lemma}\lelabel{2.6}
Let $(C,P)$ be a Dorroh pair of coalgebras and suppose that $P$ has a counit $\ep_P$. Then
the following two diagrams commute:
$$\xymatrix{
  P\ot P \ar[d]_{\rho_r\ot 1} \ar[r]^{1\ot \ep_P}
                & P \ar[d]^{\rho_r}  \\
  (P\ot C)\ot P \ar [r]_{1\ot 1\ot \ep_P}
                & P\ot C  ,           }\quad
                \xymatrix{
  P\ot P \ar[d]_{1\ot \rho_l} \ar[r]^{\ep_P \ot 1}
                &  P \ar[d]^{\rho_l}  \\
  P\ot (C\ot P) \ar [r]_{\ep_P\ot 1\ot 1}
                & C\ot P   .           }$$
\end{lemma}

\begin{proposition}\prlabel{2.4}
Let $(C,P)$ be a Dorroh pair of coalgebras and suppose that $P$ has a counit $\ep_P$. 
Then $\sum p_{(-1)}\ep_P(p_{(0)})=\sum \ep_P(p_{(0)})p_{(1)}$ for all $p\in P$.
\end{proposition}

\begin{proof}
Let $p\in P$. Then by \leref{2.6}, we have
\begin{equation*}
\begin{split}
\sum p_{(0)}\ot p_{(1)}&=\rho_r(p)=\rho_r(1\ot \ep_P)\D_P(p)=(1\ot 1\ot \ep_P)(\rho_r\ot 1)\D_P(p)\\
&=(1\ot 1\ot \ep_P)(p_{1(0)}\ot p_{1(1)}\ot p_2),\\
\sum p_{(-1)}\ot p_{(0)}&=\rho_l(p)=\rho_l(\ep_P\ot 1)
\D_P(p)=(\ep_P\ot 1\ot 1)(1\ot \rho_l)\D(p)\\
&=(\ep_P\ot 1\ot 1)(p_1\ot p_{2(-1)}\ot p_{2(0)}).
\end{split}
\end{equation*}
Thus, by Eq.\equref{e5}, one gets
\begin{equation*}
\begin{split}
\sum p_{(-1)}\ep_P(p_{(0)})&=(1\ot \ep_P)(\sum p_{(-1)}\ot p_{(0)})=(1\ot \ep_P)(\ep_P\ot 1\ot 1)(\sum p_1\ot p_{2(-1)}\ot p_{2(0)})\\
&=(\ep_P\ot 1\ot \ep_P)(\sum p_1\ot p_{2(-1)}\ot p_{2(0)})=(\ep_P\ot 1\ot \ep_P)(\sum p_{1(0)}\ot p_{1(1)}\ot p_2)\\
&=(\ep_P\ot 1)(1\ot 1\ot \ep_P)(\sum p_{1(0)}\ot p_{1(1)}\ot p_2)=(\ep_P\ot 1)(\sum p_{(0)}\ot p_{(1)})\\
&=\sum\ep_P(p_{(0)})p_{(1)}.
\end{split}
\end{equation*}
\end{proof}
Let $C\times P$ denote the direct product of the two coalgebras $C$ and $P$. The comultiplication in $C\times P$
is given by $\D(c,p)=\sum (c_1,0)\ot (c_2,0)+\sum (0,p_1)\ot (0,p_2)$ for all $(c,p)\in C\times P$.

\begin{proposition}\prlabel{prop2.5}
Let $(C,P)$ be a Dorroh pair of coalgebras. If $P$ has a counit $\ep_P$,
then $C\ltimes_d P\cong C\times P$ as coalgebras.
\end{proposition}
\begin{proof}
	Assume that $P$ has a counit $\ep_P$.
Define a map $\zeta: C\ltimes_d P\ra C\ti P$ by $\zeta(c,p)=(c,p)-(\sum p_{(-1)}\ep_P(p_{(0)}),0)$.
Let $\D_d$ and $\D_{\times}$ denote the comultiplications of $C\ltimes_d P$ and $C\times P$, respectively.
Let $p\in P$. Then by Eq.\equref{e4}, we have
\begin{equation} \eqlabel{e7}
\sum p_{1(-1)}\ep_P(p_{1(0)})\ot p_2=\sum p_{(-1)}\ep_P(p_{(0)1})\ot p_{(0)2}
=\sum p_{(-1)}\ot p_{(0)}.
\end{equation}
By \prref{2.4} and Eq.\equref{e3},
\begin{equation}\eqlabel{e8}
\begin{split}
\sum p_1\ot p_{2(-1)}\ep_P(p_{2(0)})=&\sum p_1\ot \ep_P(p_{2(0)})p_{2(1)}\\
=&\sum p_{(0)1}\ot \ep_P(p_{(0)2})p_{(1)}
=\sum p_{(0)}\ot p_{(1)}.
\end{split}
\end{equation}
Applying $\rho_l\ot\ep_P\ot 1$ on the both sides of Eq.\equref{e3}, one gets
\begin{equation*}
\begin{split}
\sum p_{1(-1)}\ot p_{1(0)}\ot \ep_P(p_{2(0)})p_{2(1)}&=\sum \rho_l(p_{(0)1})\ep_P(p_{(0)2})\ot p_{(1)}\\
&=\sum p_{(-1)}\ot p_{(0)}\ot p_{(1)}.
\end{split}
\end{equation*}
Then applying $(1\ot \ep_P\ot 1)$ to the above equation, it follows from \prref{2.4} that
\begin{equation} \eqlabel{e10}
\sum (p_{(-1)}\ep_P(p_{(0)})\ot p_{(1)}=\sum p_{1(-1)}\ep_P(p_{1(0)})\ot p_{2(-1)}\ep_P(p_{2(0)}).
\end{equation}
Now for any $c\in C$, we have
\begin{equation*}
(\zeta\ot \zeta)\D_d(c,0)=(\zeta\ot \zeta)(\sum (c_1,0)\ot (c_2,0))=\sum (c_1,0)\ot (c_2,0)=\D_{\times}\zeta(c,0).
\end{equation*}
For any $p\in P$, by Eqs.\equref{e7}-\equref{e10}, we have
\begin{equation*}
\begin{split}
(\zeta\ot\zeta)\D_d(0,p)
&=(\zeta\ot \zeta)(\sum (p_{(-1)},0)\ot (0,p_{(0)})+\sum (0,p_{(0)})\ot (p_{(1)},0)+\sum (0,p_1)\ot (0,p_2))\\
&=\sum (p_{(-1)},0)\ot (0,p_{(0)})-\sum (p_{(-2)},0)\ot (p_{(-1)}\ep_P(p_{(0)}),0)\\
&\quad +\sum (0,p_{(0)})\ot (p_{(1)},0)-\sum (p_{(-1)}\ep_P(p_{(0)}),0)\ot (p_{(1)},0)\\
&\quad +\sum (0,p_1)\ot (0,p_2)-\sum (p_{1(-1)}\ep_P(p_{1(0)}),0)\ot (0,p_2)\\
&\quad-\sum (0,p_1)\ot (p_{2(-1)}\ep_P(p_{2(0)}),0)+\sum (p_{1(-1)}\ep_P(p_{1(0)}),0)\ot (p_{2(-1)}\ep_P(p_{2(0)}),0)\\
&=\sum (0,p_1)\ot (0,p_2)-\sum(p_{(-2)},0)\ot (p_{(-1)}\ep_P(p_{(0)}),0)\\
&=\D_{\times}((0,p)-(\sum p_{(-1)}\ep_P(p_{(0)}),0))=\D_{\times}\zeta(0,p).
\end{split}
\end{equation*}
Therefore, $\zeta$ is a coalgebra homomorphism.
It is easy to check that $\zeta$ is a bijection. Hence $\zeta$ is a coalgebra isomorphism.
\end{proof}

Let $\mathcal{C}$ denote the category of all coalgebras and coalgebra morphisms.

For two coalgebra Dorroh pairs $(C,P)$ and $(C',P')$, we call a pair $(\phi, f): (C,P)\ra(C',P')$
{\it a homomorphism of Dorroh pair of coalgebras} 
if $\phi: C\ra C'$ and $f:P\ra P'$ are coalgebra homomorphisms satisfying
 $\sum f(p)_{(-1)}\ot f(p)_{(0)}=\sum\phi(p_{(-1)})\ot f(p_{(0)})$ and $\sum f(p)_{(0)}\ot f(p)_{(1)}=\sum f(p_{(0)})\ot \phi(p_{(1)})$
for all $p\in P$. Let $\mathcal{DPC}$ be the category of all Dorroh pairs of coalgebras
and homomorphisms.


For any coalgebra $D$, $(D,D)$ is a Dorroh pair of coalgebras with the regular comodule structures. 

\begin{remark}\relabel{rem2.9}
Let $(C,P)$ be a Dorroh pair of coalgebras. Then the projections $\pi_C: C\ltimes_d P \ra C $ 
and $\pi_P: C\ltimes_d P\ra P$ are coalgebra homomorphisms. Moreover,
$\sum\pi_P(c,p)_{(-1)}\ot \pi_P(c,p)_{(0)}=\sum p_{(-1)}\ot p_{(0)}=(\pi_C\ot \pi_P)(\D(c,p))$
and $\sum\pi_P(c,p)_{(0)}\ot \pi_P(c,p)_{(1)}=\sum p_{(0)}\ot p_{(1)}=(\pi_P\ot \pi_C)(\D(c,p))$
for all $(c, p)\in C\ltimes_d P$. That is, $(\pi_C, \pi_P)$ is
a Dorroh pair homomorphism from $(C\ltimes_d P, C\ltimes_d P)$ to $(C,P)$.
\end{remark}

The following result describes that Dorroh extension $C\ltimes_d P$ and the projections
$\pi_C$, $\pi_P$ have a universal property in the category $\mathcal{C}$.

\begin{proposition}
Let $(C,P)$ be a Dorroh pair of coalgebras and $D$ a coalgebra. Suppose that $(\vf, f): (D, D)\ra (C, P)$ 
is a homomorphism of Dorroh pair of coalgebras. Then there is a unique coalgebra homomorphism 
$\eta: D\ra C\ltimes_d P$ such that $\pi_C  \eta=\vf$ and $\pi_P \eta=f$.
\end{proposition}

\begin{proof}
Define $\eta: D\ra C\ltimes_dP$ by $\eta(d)=(\vf(d),f(d))$ for $d\in D$. Then we have
\begin{equation*}
\begin{split}
\D \eta(d)=&\sum (\vf(d)_1,0)\ot (\vf(d)_2,0)+\sum (f(d)_{(-1)},0)\ot (0,f(d)_{(0)})\\
&+\sum (0,f(d)_{(0)})\ot (f(d)_{(1)},0)+\sum(0,f(d)_1)\ot (0,f(d)_2)\\
=&\sum (\vf(d_1),0)\ot (\vf(d_2),0)+\sum (\vf(d_1),0)\ot (0,f(d_2))\\
&+\sum (0,f(d_1))\ot (\vf(d_2),0)+\sum(0,f(d_1))\ot (0,f(d_2))\\
=&\sum(\vf(d_1),f(d_1))\ot (\vf(d_2), f(d_2))=(\eta \ot \eta)\D(d).\\
\end{split}
\end{equation*}
Hence $\eta $ is a coalgebra homomorphism. Obviously, $\pi_C \eta=\vf$ and $\pi_P \eta=f$.
Now assume that $\eta': D\ra C\ltimes_dP$ is another coalgebra homomorphism such that
$\pi_C \eta'(d)=\vf(d)$ and $\pi_P\eta'(d)=f(d)$. Then $\eta'(d)=(\pi_C( \eta'(d)),\pi_P(\eta'(d)))=(\vf(d),f(d))=\eta(d)$.
This shows the uniqueness of $\eta$.
\end{proof}

Now we give some examples of Dorroh pairs of coalgebras and Dorroh extensions of coalgebras.

\begin{example}\exlabel{ex2.10}
(a) Let $P$ be a coalgebra. Then $P$ is a trivial $k$-bicomodule
and $(k,P)$ is a Dorroh pair of coalgebras. Hence
$k\ltimes_d P$ is a counital coalgebra with the comultiplication given by
$$\D(\a, p)=(\a,0)\ot (1,0)+(1,0)\ot (0,p)+(0,p)\ot (1,0)+\sum(0,p_1)\ot (0,p_2),$$
and the counit $\varepsilon$ given by  $\varepsilon(\a, p)=\a$, $(\a, p)\in k\ltimes_dP$.

(b) Let $C$ and $P$ be coalgebras. Regarding $P$ as a $C$-bicomodule with the zero coactions, i.e.,
$\rho_l=\rho_r=0$. Then $(C,P)$ is a Dorroh pair of coalgebras 
and $C\ltimes_d P$ is exactly the direct product coalgebra $C\ti P$ of $C$ and $P$. The comultiplication
is given by $\D(c,p)=\sum (c_1,0)\ot (c_2,0)+\sum (0,p_1)\ot (0,p_2)$ for all $(c,p)\in C\times P$.

(c) Let $C$ be a coalgebra and $M$ a  $C$-bicomodule. One can construct the trivial extension $C\ltimes M$
as follows:  $C\ltimes M=C\op M$ as vector spaces, the comultiplication is defined by 
$\D(c,m)=\sum (c_1,0)\ot (c_2,0)+\sum (m_{(-1)},0)\ot (0,m_{(0)})+\sum (0,m_{(0)})\ot (m_{(1)},0)$ for all
$c\in C$ and $m\in M$. If we define a comultiplication of $M$ by $\D_M=0$,
then $(C,M)$ is a Dorroh pair of coalgebras and $C\ltimes_d M$ is exactly the trivial extension $C\ltimes M$.

(d) Assume that $C$ and $D$ are coalgebras. Let $M$ be a $C$-$D$-bicomodule. 
Then one can construct a triangular matrix coalgebra 
$\left(\begin{array}{cc}
 C & M \\
 0 & D \\
\end{array}\right)$.
 Note that $M$ is also a  $C\times D$-bicomodule with
 $\rho_l(m):=\sum(m_{(-1)},0)\ot m_{(0)}$ and
 $\rho_r(m):=\sum m_{(0)}\ot(0,m_{(1)})$.
 Then the trivial extension $(C\times D)\ltimes M$ is isomorphic to the triangular matrix coalgebra
$\left(\begin{array}{cc}
C & M \\
0 & D \\
\end{array}\right)$. 
In particular, one-point extension $\left(
                   \begin{array}{cc}
                     C & M \\
                     0 & k \\
                   \end{array}
                 \right)$ is a Dorroh extension of coalgebra.

(e) Consider the multiplicative monoid $\mathbb{Z}_2=\{0,1\}$, where $0\cdot 0=0\cdot 1=1\cdot 0=0,1\cdot 1=1$.
Let $D_1:=C$ and $D_0:=P$ be embedded into Dorroh extension $D= C\ltimes_d P$ as before. 
Then we have $\D(D_1)\ss D_1\ot D_1$ and $\D(D_0)\ss D_0\ot D_0+D_0\ot D_1+D_1\ot D_0$. 
Hence $C\ltimes_d P$ is a $\mathbb{Z}_2$-graded coalgebra.
Conversely, any $\mathbb{Z}_2$-graded coalgebra $C=C_1\op C_0$ is a Dorroh extension of $C_1$ by $C_0$.

(f) Let $D=\oplus_{i=0}^{\infty} D_i$ be an $\mathbb{N}$-graded coalgebra.
Set $P:=\oplus_{i=1}^{\infty} D_i$. Then $D$ is a Dorroh extension of $D_0$ by $P$.
\end{example}

In the following, we consider modules over $C\ltimes_d P$ and iterated Dorroh extension of coalgebras.

Let $(C,P)$ be a Dorroh pair and $M$ a vector space.
Let $\rho_l$ and $\rho_r$ denote the left and right $C$-comodule structure maps of $P$, respectively.
If $(M, \rho_l^{C\ltimes_d P})$ is a left $C\ltimes_d P$-comodule, then $M$ is both a left $C$-comodule
via $\rho_l^C:=(\pi_C\ot 1)\rho_l^{C\ltimes_d P}$ and a left $P$-comodule via $\rho_l^P:=(\pi_P\ot 1)\rho_l^{C\ltimes_d P}$
since $\pi_C: C\ltimes_d P\ra C$ and $\pi_P: C\ltimes_d P\ra P$ are both coalgebra homomorphisms.
By \reref{rem2.9}, $(\pi_P\ot \pi_C)\D=\rho_r\pi_P$ and $(\pi_C\ot \pi_P)\D=\rho_l\pi_P$. Hence we have
\begin{equation*}
\begin{split}
(1\ot \rho_l^C)\rho_l^P&=(1\ot (\pi_C\ot 1)\rho_l^{C\ltimes_d P})(\pi_P\ot 1)\rho_l^{C\ltimes_d P}
=(\pi_P\ot \pi_C\ot 1)(1\ot \rho_l^{C\ltimes_d P})\rho_l^{C\ltimes_d P}\\
&=(\pi_P\ot \pi_C\ot 1)(\D\ot 1)\rho_l^{C\ltimes_d P}
=(\rho_r\pi_P\ot 1)\rho_l^{C\ltimes_d P}=(\rho_r\ot 1)(\pi_P\ot 1)\rho_l^{C\ltimes_d P}\\
&=(\rho_r\ot 1)\rho_l^P,
\end{split}
\end{equation*}
and similarly $(1\ot \rho_l^P)\rho_l^C=(\rho_l\ot 1)\rho_l^P$.

Conversely, suppose that $M$ is both a left  $C$-comodule via $\rho_l^C$ and a left $P$-comodule via
$\rho_l^P$ satisfying $(1\ot \rho_l^C)\rho_l^P=(\rho_r\ot 1)\rho_l^P$ and
$(1\ot \rho_l^P)\rho_l^C=(\rho_l\ot 1)\rho_l^P$. For $m\in M$, we write
$\rho^C_l(m)=\sum m_{(-1)}\ot m_{(0)}$ and $\rho^P_l(m)=\sum m_{[-1]}\ot m_{[0]}$.
Then $(1\ot \rho_l^C)\rho_l^P=(\rho_r\ot 1)\rho_l^P$ and $(1\ot \rho_l^P)\rho_l^C=(\rho_l\ot 1)\rho_l^P$
are equivalent to $\sum m_{[-1]}\ot m_{[0](-1)}\ot m_{[0](0)}=\sum m_{[-1](0)}\ot m_{[-1](1)}\ot m_{[0]}$ and
$\sum m_{(-1)}\ot m_{(0)[-1]}\ot m_{(0)[0]}=\sum m_{[-1](-1)}\ot m_{[-1](0)}\ot m_{[0]}$, respectively, where $m\in M$.

Now define $\rho_l^{C\ltimes_d P}: M\ra (C\ltimes_d P)\ot M$ by
$\rho_l^{C\ltimes_d P}(m)=\sum (m_{(-1)},0)\ot m_{(0)}+\sum (0, m_{[-1]})\ot m_{[0]}$ for $m\in M$.
Then for any $m\in M$, we have
\begin{equation*}
\begin{split}
(1\ot \rho_l^{C\ltimes_d P})\rho_l^{C\ltimes_d P}(m)
=&(1\ot \rho_l^{C\ltimes_d P})(\sum (m_{(-1)},0)\ot m_{(0)}+\sum (0, m_{[-1]})\ot m_{[0]})\\
=&\sum(m_{(-2)},0)\ot (m_{(-1)},0)\ot m_{(0)}
+\sum(m_{(-1)},0)\ot (0, m_{(0)[-1]})\ot m_{(0)[0]}\\
&+\sum(0, m_{[-1]})\ot (m_{[0](-1)},0)\ot m_{[0](0)}
+\sum(0,m_{[-2]})\ot (0, m_{[-1]})\ot m_{[0]}\\
=&\sum(m_{(-2)},0)\ot (m_{(-1)},0)\ot m_{(0)}
+\sum(m_{[-1](-1)},0)\ot (0, m_{[-1](0)})\ot m_{[0]}\\
&+\sum(0, m_{[-1](0)})\ot (m_{[-1](1)},0)\ot m_{[0]}
+\sum(0,m_{[-2]})\ot (0 ,m_{[-1]})\ot m_{[0]}\\
=&(\D\ot 1)(\sum (m_{(-1)},0)\ot m_{(0)}+\sum (0, m_{[-1]})\ot m_{[0]})\\
=&(\D\ot 1)\rho_l^{C\ltimes_d P}(m).
\end{split}
\end{equation*}
Hence $(M, \rho_l^{C\ltimes_d P})$ is a left $C\ltimes_d P$-comodule.

For right comodules, there is a similar statement. We conclude these in the following.

\begin{proposition}\prlabel{prop2.12}
Let $(C,P)$ be a Dorroh pair of coalgebras and $M$ a vector space. Then $M$ is a left (resp., right) $C\ltimes_d P$-comodule
if and only if $M$ is both a left (resp., right) $C$-comodule via $\rho_l^C$ (resp., $\rho_r^C$)
and a left (resp., right) $P$-comodule via $\rho_l^P$ (resp., $\rho_r^P$) such that
$(1\ot \rho_l^C)\rho_l^P=(\rho_r\ot 1)\rho_l^P$ and $(1\ot \rho_l^P)\rho_l^C=(\rho_l\ot 1)\rho_l^P$
(resp., $(\rho_r^C\ot 1)\rho_r^P=(1\ot\rho_l)\rho_r^P$ and $(\rho_r^P\ot 1)\rho_r^C=(1\ot\rho_r)\rho_r^P$).

If this is the case, then $M$ is a $C\ltimes_d P$-bicomodule
if and only if $M$ is meanwhile a  $C$-bicomodule, a $P$-bicomodule, a $C$-$P$-bicomodule and a $P$-$C$-bicomodule.
\end{proposition}
\begin{proof}
We only need to show the last statement of the lemma. The ``only if" part is obvious
since the projections $\pi_C: C\ltimes_d P\ra C$ and $\pi_P: C\ltimes_d P\ra P$
are both coalgebra homomorphisms. Now we show the ``if" part.

Assume that $M$ is a  $C$-bicomodule, a $P$-bicomodule, a $C$-$P$-bicomodule and a $P$-$C$-bicomodule
 with the comodule structures  $\rho_l^C$, $\rho_l^P$, $\rho_r^C$ and $\rho_r^P$.
We use the notations before and write $\rho_r^C(m)=\sum m_{(0)}\ot m_{(1)}$ and
$\rho_r^P(m)=\sum m_{[0]}\ot m_{[1]}$ for $m\in M$. Then the right comodule structure map
$\rho_r^{C\ltimes_d P}: M\ra M\ot C\ltimes_d P$ is given by
$\rho_r^{C\ltimes_d P}(m)=\sum m_{(0)}\ot(m_{(1)}, 0)+\sum m_{[0]}\ot (0, m_{[1]})$ for $m\in M$.
Hence for any $m\in M$, we have
\begin{equation*}
\begin{split}
(1\ot \rho_r^{C\ltimes_d P})\rho_l^{C\ltimes_d P}(m)
=&(1\ot \rho_r^{C\ltimes_d P})\Big(\sum (m_{(-1)},0)\ot m_{(0)}+\sum (0, m_{[-1]})\ot m_{[0]}\Big)\\
=&\sum (m_{(-1)},0)\ot m_{(0)}\ot (m_{(1)},0)+\sum (m_{(-1)},0)\ot m_{(0)[0]}\ot (0, m_{(0)[1]})\\
&+\sum (0,m_{[-1]})\ot m_{[0](0)}\ot (m_{[0](1)},0)+\sum (0, m_{[-1]})\ot m_{[0]}\ot (0,m_{[1]})\\
=&\sum (m_{(-1)},0)\ot m_{(0)}\ot (m_{(1)},0)+\sum (m_{[0](-1)},0)\ot m_{[0](0)}\ot (0, m_{[1]})\\
&+\sum (0,m_{(0)[-1]})\ot m_{(0)[0]}\ot (m_{(1)},0)+\sum (0, m_{[-1]})\ot m_{[0]}\ot (0,m_{[1]})\\
=&(\rho_l^{C\ltimes_d P}\ot 1)\Big(\sum m_{(0)}\ot(m_{(1)}, 0)+\sum m_{[0]}\ot (0, m_{[1]})\Big)\\
=&(\rho_l^{C\ltimes_d P}\ot 1)\rho_r^{C\ltimes_d P}(m).
\end{split}
\end{equation*}
This shows that $(M,\rho_l^{C\ltimes_d P},\rho_r^{C\ltimes_d P})$ is a $C\ltimes_d P$-bicomodule.
\end{proof}

\begin{lemma}\lelabel{lem2.13}
Let $(C, P)$ be a Dorroh pair of coalgebras and $D$ be a coalgebra. If there is a coalgebra homomorphism $f: C\ra D$,
then $(D, P)$ is also a Dorroh pair of coalgebras, where the left and right $D$-comodule structure maps of $P$ are given
by $\rho^D_l=(f\ot 1)\rho_l$ and $\rho^D_r=(1\ot f)\rho_r$, respectively.
\end{lemma}
\begin{proof}
Since $P$ is a $C$-bicomodule and $f$ is a coalgebra homomorphism, $P$ is a $D$-bicomodule via
$\rho^D_l$ and $\rho^D_r$. Since $(C, P)$ is a Dorroh pair, we have
$$\begin{array}{rl}
(\D_P\ot 1)\rho^D_r=&(\D_P\ot 1)(1\ot f)\rho_r=(1\ot1\ot f)(\D_P\ot 1)\rho_r\\
=&(1\ot1\ot f)(1\ot\rho_r)\D_P=(1\ot\rho^D_r)\D_P.\\
\end{array}$$
Similarly, one can check that $(1\ot\D_P)\rho^D_l=(\rho^D_l\ot 1)\D_P$ and
$(\rho^D_r\ot 1)\D_P=(1\ot\rho^D_l)\D_P$. Hence $(D, P)$ is a Dorroh pair of coalgebras.	
\end{proof}

Let $(C_1,C_2),(C_1,C_3),(C_2,C_3)\in \mathcal{DPC}$. For $1\<i<j\<3$,
the left and right $C_i$-comodule structure maps of $C_j$ are denoted by $\rho_{l, C_j}^{C_i}$
and $\rho_{r, C_j}^{C_i}$, respectively. And we write $\rho_{l, C_j}^{C_i}(c)=\sum c_{(-1)}^{(i)}\ot c_{(0)}^{(i)}$
and $\rho_{r, C_j}^{C_i}(c)=\sum c_{(0)}^{(i)}\ot c_{(1)}^{(i)}$ for $c\in C_j$.

Define linear maps
$\rho_{l,C_3}^{C_1\ltimes_d C_2}: C_3\ra (C_1\ltimes_d C_2)\ot C_3$
and $\rho_{r,C_3}^{C_1\ltimes_d C_2}: C_3\ra C_3\ot(C_1\ltimes_d C_2)$ by
\begin{equation*}
\begin{split}
\rho_{l,C_3}^{C_1\ltimes_d C_2}(d)&=\sum (d^{(1)}_{(-1)},0)\ot d^{(1)}_{(0)}+\sum (0,d^{(2)}_{(-1)})\ot d^{(2)}_{(0)},\\
\rho_{r,C_3}^{C_1\ltimes_d C_2}(d)&=\sum d^{(1)}_{(0)}\ot(d^{(1)}_{(1)},0)+\sum d^{(2)}_{(0)}\ot (0, d^{(2)}_{(1)}),
\end{split}
\end{equation*}
and
$\rho_{l,C_2\ltimes_d C_3}^{C_1}: C_2\ltimes_d C_3\ra C_1\ot (C_2\ltimes_d C_3)$ and
$\rho_{r,C_2\ltimes_d C_3}^{C_1}: C_2\ltimes_d C_3\ra (C_2\ltimes_d C_3)\ot C_1$ by
\begin{equation*}
\begin{split}
\rho_{l,C_2\ltimes_d C_3}^{C_1}(c,d)&=\sum c^{(1)}_{(-1)}\ot (c^{(1)}_{(0)},0)+\sum d^{(1)}_{(-1)}\ot (0,d^{(1)}_{(0)}),\\
\rho_{r,C_2\ltimes_d C_3}^{C_1}(c,d)&=\sum(c^{(1)}_{(0)},0)\ot c^{(1)}_{(1)}+\sum(0,d^{(1)}_{(0)})\ot d^{(1)}_{(1)}
\end{split}
\end{equation*}
for $c\in C_2$ and $d\in C_3$.
Then $\rho_{l,C_3}^{C_1\ltimes_d C_2}=(\iota_{1}\ot 1)\rho_{l,C_3}^{C_1}+(\iota_{2}\ot 1)\rho_{l,C_3}^{C_2}$ and
$\rho_{r,C_3}^{C_1\ltimes_d C_2}=(1\ot\iota_{1})\rho_{r,C_3}^{C_1}+(1\ot\iota_{2})\rho_{r,C_3}^{C_2}$,
where $\iota_{i}:C_i\ra C_1\ltimes_d C_2$ are the canonical embeddings, $i=1, 2$.

\begin{proposition}\prlabel{prop2.14}
Let $C_1,C_2,C_3$ be coalgebras and assume $(C_1,C_2)\in \mathcal{DPC} $. Then $(C_1\ltimes_d C_2, C_3)\in \mathcal{DPC}$
if and only if $(C_1,C_3),(C_2,C_3) \in \mathcal{DPC} $,
$C_3$ is both a $C_1$-$C_2$-bicomodule and a $C_2$-$C_1$-bicomodule such that the following equations are satisfied:
  \begin{eqnarray}
   \eqlabel{e11}(1\ot \rho_{l,C_3}^{C_1})\rho_{l,C_3}^{C_2}&=&(\rho_{r,C_2}^{C_1}\ot 1)\rho_{l,C_3}^{C_2}, \\
  \eqlabel{e12} (1\ot \rho_{l,C_3}^{C_2})\rho_{l,C_3}^{C_1}&=&(\rho_{l,C_2}^{C_1}\ot 1)\rho_{l,C_3}^{C_2}, \\
  \eqlabel{e13 } (\rho_{r,C_3}^{C_1}\ot 1)\rho_{r,C_3}^{C_2}&=&(1\ot\rho_{l,C_2}^{C_1})\rho_{r,C_3}^{C_2},\\
  \eqlabel{e14} (\rho_{r,C_3}^{C_2}\ot 1)\rho_{r,C_3}^{C_1}&=&(1\ot\rho_{r,C_2}^{C_1})\rho_{r,C_3}^{C_2}.
\end{eqnarray}

If this is the case, then $(C_1, C_2 \ltimes_d C_3)\in \mathcal{DPC}$ and 
$(C_1\ltimes_d C_2)\ltimes_d C_3\cong C_1\ltimes_d(C_2 \ltimes_d C_3)$ as coalgebras.
\end{proposition}
\begin{proof}
The ``only if" part of the first statement follows from \prref{prop2.12} and \leref{lem2.13}.
Now we show its ``if" part.
By \prref{prop2.12}, $C_3$ is a $C_1\ltimes_d C_2$-bicomodule with the comodule structure maps
$\rho_{l,C_3}^{C_1\ltimes_d C_2}$ and $\rho_{r,C_3}^{C_1\ltimes_d C_2}$ defined above.
Moreover, we have
$$\begin{array}{rl}
(\D_{C_3}\ot 1)\rho_{r,C_3}^{C_1\ltimes_d C_2}
=&(\D_{C_3}\ot 1)((1\ot\iota_{1})\rho_{r,C_3}^{C_1}+(1\ot\iota_{2})\rho_{r,C_3}^{C_2})\\
=&(\D_{C_3}\ot 1)(1\ot\iota_{1})\rho_{r,C_3}^{C_1}+(\D_{C_3}\ot 1)(1\ot\iota_{2})\rho_{r,C_3}^{C_2}\\
=&(1\ot1\ot\iota_{1})(\D_{C_3}\ot 1)\rho_{r,C_3}^{C_1}+(1\ot1\ot\iota_{2})(\D_{C_3}\ot 1)\rho_{r,C_3}^{C_2}\\
=&(1\ot1\ot\iota_{1})(1\ot \rho_{r,C_3}^{C_1})\D_{C_3}+(1\ot1\ot\iota_{2})(1\ot \rho_{r,C_3}^{C_2})\D_{C_3}\\
=&(1\ot(1\ot\iota_{1})\rho_{r,C_3}^{C_1})\D_{C_3}+(1\ot(1\ot\iota_{2})\rho_{r,C_3}^{C_2})\D_{C_3}\\
=&(1\ot(1\ot\iota_{1})\rho_{r,C_3}^{C_1}+1\ot(1\ot\iota_{2})\rho_{r,C_3}^{C_2})\D_{C_3}\\
=&(1\ot \rho_{r,C_3}^{C_1\ltimes_d C_2})\D_{C_3}.
\end{array}$$
Similarly, one can check that $(1\ot\D_{C_3})\rho_{l,C_3}^{C_1\ltimes_d C_2}=(\rho_{l,C_3}^{C_1\ltimes_d C_2}\ot 1)\D_{C_3}$
and $(\rho_{r,C_3}^{C_1\ltimes_d C_2}\ot 1)\D_{C_3}=(1\ot\rho_{l,C_3}^{C_1\ltimes_d C_2})\D_{C_3}$.
Hence $(C_1\ltimes_d C_2, C_3)\in \mathcal{DPC}$.  This shows the first statement of the proposition.

Now we show the second statement. By the hypotheses, $C_2$ and $C_3$ are both $C_1$-bicomodule.
Since $C_2\ltimes_d C_3=C_2\oplus C_3$ as vector spaces,
$C_2\ltimes_d C_3$ becomes naturally a $C_1$-bicomodule with the left and right comodule structure maps
$\rho_{l,C_2\ltimes_dC_3}^{C_1}$ and $\rho_{r,C_2\ltimes_dC_3}^{C_1}$ given above.
For any $c\in C_2$ and $d\in C_3$, we have
\begin{equation*}
\begin{split}
(\D_{C_2\ltimes_d C_3}\ot 1)\rho_{r,C_2\ltimes_d C_3}^{C_1}(c,d)
=&(\D_{C_2\ltimes_d C_3}\ot 1)(\sum (c_{(0)}^{(1)}, 0)\ot c_{(1)}^{(1)}
+\sum (0, d_{(0)}^{(1)})\ot d_{(1)}^{(1)})\\
=&\sum (c_{(0)1}^{(1)},0)\ot (c_{(0)2}^{(1)},0)\ot c_{(1)}^{(1)}
+\sum (d_{(0)(-1)}^{(1)(2)},0)\ot (0,d_{(0)(0)}^{(1)(2)})\ot d_{(1)}^{(1)}\\
&+\sum (0,d_{(0)(0)}^{(1)(2)})\ot (d_{(0)(1)}^{(1)(2)},0)\ot d_{(1)}^{(1)}
+\sum (0,d_{(0)1}^{(1)})\ot (0,d_{(0)2}^{(1)})\ot d_{(1)}^{(1)}\\
=&\sum (c_{1},0)\ot (c_{2(0)}^{(1)},0)\ot c_{2(1)}^{(1)}
+\sum (d_{(-1)}^{(2)},0)\ot (0,d_{(0)(0)}^{(2)(1)})\ot d_{(0)(1)}^{(2)(1)}\\
&+\sum (0,d_{(0)}^{(2)})\ot (d_{(1)(0)}^{(2)(1)},0)\ot d_{(1)(1)}^{(2)(1)}
+\sum (0,d_{1})\ot (0,d_{2(0)}^{(1)})\ot d_{2(1)}^{(1)}\\
=&(1\ot\rho_{r,C_2\ltimes_d C_3}^{C_1})(\sum (c_{1},0)\ot (c_{2},0)
+\sum (d_{(-1)}^{(2)},0)\ot (0,d_{(0)}^{(2)})\\
&\quad\quad\quad\quad\quad\quad\quad+\sum (0,d_{(0)}^{(2)})\ot (d_{(1)}^{(2)},0)
+\sum (0,d_{1})\ot (0,d_{2}))\\
=&(1\ot\rho_{r,C_2\ltimes_d C_3}^{C_1})\D_{C_2\ltimes _d C_3}(c,d).\\
\end{split}
\end{equation*}
This shows $(\D_{C_2\ltimes_d C_3}\ot 1)\rho_{r,C_2\ltimes_d C_3}^{C_1}
=(1\ot\rho_{r,C_2\ltimes_d C_3}^{C_1})\D_{C_2\ltimes _d C_3}$.
Similarly, one can show
$(1\ot \D_{C_2\ltimes_d C_3})\rho_{l,C_2\ltimes_d C_3}^{C_1}=(\rho_{l,C_2\ltimes_d C_3}^{C_1}\ot 1)\D_{C_2\ltimes _d C_3}$
and
$(\rho_{r,C_2\ltimes_d C_3}^{C_1}\ot 1)\D_{C_2\ltimes_d C_3}=(1\ot\rho_{l,C_2\ltimes_d C_3}^{C_1})\D_{C_2\ltimes _d C_3}$.
Therefore, $(C_1, C_2 \ltimes_d C_3)\in \mathcal{DPC}$.
It is straightforward to check that the map
$$(C_1\ltimes_d C_2) \ltimes_d C_3\ra C_1\ltimes_d (C_2\ltimes_d C_3),\
((c,d),e)\mapsto (c,(d,e))$$
is a  coalgebra isomorphism.
\end{proof}

\begin{remark}
For the two coalgebras $(C_1\ltimes_d C_2) \ltimes_d C_3$ and $C_1\ltimes_d (C_2\ltimes_d C_3)$ in \prref{prop2.14}, 
we can simply write $C_1\ltimes_d C_2\ltimes_d C_3$.
More generally, one can define a Dorroh $n$-tuple of coalgebras $(C_1,C_2,\cdots, C_n)$,
and form an iterated coalgebra Dorroh extension $C_1\ltimes_d C_2\ltimes_d \cdots \ltimes_d C_n$.
\end{remark}

\section{The relations between Dorroh extensions of coalgebras and algebras}
In this section, we firstly discuss the finite duals of algebras and modules.
We don't assume that an algebra has an identity. Hence the finite
dual of an algebra is different from that of a unital algebra. Then we give the relations between
Dorroh extension of algebras and Dorroh extension of coalgebras. Lastly we also describe
the finite duals of unital modules over a unital algebra.

Let $A$ be an algebra and $A^*=\Hom(A,k)$, the dual space of $A$.
Then $A^*$ is a right (resp., left) $A$-module defined by $(f\lhu a)(x)=f(ax)$ (resp., $(a\rhu f)(x)=f(xa)$)
for all $a, x\in A$ and $f\in A^*$. Moreover, $A^*$ is an $A$-bimodule.

Define $A^{\circ}:=\{f\in A^*\mid f(R)=0\text{ for some right ideal } R \text{ of } A \text{ of finite codimension}\}$.
Clearly, $A^{\circ}$ is a subspace of $A^*$ since the intersection of two right ideals of finite codimension
has still finite codimension. $A^{\circ}$ is called {\it the finite dual of $A$}.

\begin{proposition}\prlabel{prop3.1}
Let $(A,m)$ be an algebra and $f\in A^*$. Then the following are equivalent:
\begin{enumerate}
\item[(a)] $f$ vanishes on a right ideal of $A$ of finite codimension;
\item[(b)] $f$ vanishes on a left ideal of $A$ of finite codimension;
\item[(c)] $f(RA)=0$ for some right ideal $R$ of $A$ of finite codimension;
\item[(d)] $f(AL)=0$ for some left ideal $L$ of $A$ of finite codimension;
\item[(e)] $\dim (A\rhu f)<\infty$;
\item[(f)] $\dim ( f\lhu A)<\infty$;
\item[(g)] $m^*(f)\in A^*\ot A^*$.
\end{enumerate}
\end{proposition}
\begin{proof}
(a)$\Rightarrow$(c) and (b)$\Rightarrow$(d) are obvious. Now we show
(c)$\Rightarrow$(e) and (e)$\Rightarrow$(a).  Firstly, let $R$ be a right ideal of $A$ of finite codimension such that $f(RA)=0$.
Define $R^{\perp}:=\{g\in A^*\mid g(R)=0\}$. Then $R^{\perp}$ is finite dimensional.
For any $a\in A$, $(a\rhu f)(R)=f(Ra)=0$. Thus, $A\rhu f\ss R^{\perp}$, and so $A\rhu f$ is finite dimensional.
This shows (c)$\Rightarrow$(e). Next, suppose $\dim (A\rhu f)<\infty$. Then $A\rhu f+kf$ is a finite dimensional
left $A$-submodule of $A^*$. Let $R=(A\rhu f+kf)^{\perp}:=\{a\in A|g(a)=0, \forall g\in (A\rhu f+kf)\}$.
Then $R$ is a right ideal of $A$ with finite codimension and $f(R)=0$. This shows (e)$\Rightarrow$(a).
Similarly, one can show (d)$\Rightarrow$(f) and (f)$\Rightarrow$(b).

(e)$\Leftrightarrow$(g) Suppose  $\dim (A\rhu f)<\infty$. Let $f_1,f_2,\cdots, f_n$ be a basis of $A\rhu f$. For any $a\in A$,
$a\rhu f=\sum_{i=1}^n g_i(a)f_i$ for some $g_i(a)\in k$. Clearly, $g_i\in A^*$ for all $1\<i\<n$.
For any $a, b\in A$,
$m^*(f)(a\ot b)=f(ab)=(b\rhu f)(a)=\sum_{i=1}^ng_i(b)f_i(a)=(\sum_{i=1}^n f_i\ot g_i)(a\ot b)$.
It follows that $m^*(f)=\sum_{i=1}^n f_i\ot g_i\in A^*\ot A^*$. Conversely, suppose $m^*(f)\in A^*\ot A^*$.
Then $m^*(f)=\sum_{i=1}^n f_i\ot g_i$ for some $f_i, g_i\in A^*$. Hence
for any $a, b\in A$, we have
$(a\rhu f)(b)=f(ba)=m^*(f)(b\ot a)=\sum_{i=1}^nf_i(b)g_i(a)=(\sum_{i=1}^ng_i(a)f_i)(b)$.
Thus, $a\rhu f=\sum_{i=1}^n g_i(a)f_i$ for any $a\in A$,
and so $A\rhu f$ is finite dimensional.

Similarly, one can show $(f)\Leftrightarrow (g)$.
\end{proof}

\begin{remark}
Since $A$ is not necessarily unital, $\dim (A\rhu f)<\infty $ implies  $\dim (A\rhu f\lhu A)<\infty $,
but the converse is not true. This is different from \cite[Lemma 9.1.1]{Mo}.
\end{remark}

\begin{proposition}\prlabel{prop3.3}
Let $(A,m)$ be an algebra. Then $m^*(A^{\circ})\ss A^{\circ}\ot A^{\circ}$. Thus, $(A^{\circ}, m^*)$ is a coalgebra.
\end{proposition}
\begin{proof}
Let $0\neq f\in A^{\circ}$. Similarly to the proof of \cite[Proposition 9.1.2]{Mo},
let $\{f_1,f_2,\cdots, f_n\}$ be a basis of $A\rhu f$. Then $m^*(f) =\sum f_i\ot g_i$
for some $g_i\in A^*$. Since each $f_i\in A\rhu f$, there exist $a_1,a_2,\cdots , a_n\in A$
such that $f_i=a_i\rhu f$. Then $A\rhu f_i=A\rhu (a_i\rhu f)=(Aa_i)\rhu f\ss A\rhu f$.
This implies that $A\rhu f_i$ is finite dimensional, and so
$f_i\in A^{\circ}$. Since $f_1,f_2,\cdots , f_n$ are linearly independent, we may choose
$a_1',a_2',\cdots, a_n'\in A$ such that $f_i(a_j')=\d_{ij}$ by \cite[Lemma 2.2.9]{Abe}.
Then $f\lhu a_j'=\sum f_i(a_j')g_i=g_j$, and hence $g_j\in f\lhu A$. Thus,
$g_j\lhu A=(f\lhu a_j')\lhu A=f\lhu (a_j'A)\ss f\lhu A$, which is finite dimensional. So $g_j \in A^{\circ}$.

The coassociativity of $\D=m^*$ follows by the associativity of $m$.
\end{proof}

Let $M$ be a left (resp., right) $A$-module. Let $M^*=\Hom(M,k)$ be the dual space of $M$.
Then $M^*$ is a right (resp. left) $A$-module with $(f\lhu a)(x)=f(ax)$ (resp., $(a\rhu f)(x)=f(xa)$)
for all $a\in A$, $x\in M$, $f\in M^*$. Moreover, if $M$ is an $A$-bimodule then so is $M^*$.

For a right (or left) $A$-module $M$, define
$$M_{r(A)}^{\circ} (\text{ or }M_{l(A)}^{\circ}):=\{f\in M^*\mid f(N)
=0 \text{ for some submodule }N \text{ of } M \text{ with } \text{dim}(M/N)<\infty\}.$$
$M_{r(A)}^{\circ}$ (or $M_{l(A)}^{\circ}$) is called {\it the finite dual of the right (or left) $A$-module $M$},
and written as $M_r^{\circ}$ (or $M_l^{\circ}$) simply if there is no ambiguity.

\begin{proposition}\prlabel{prop3.4}
Let $M$ be a right $A$-module and $f\in M^*$. Then the following are equivalent:
\begin{enumerate}
\item[(a)] $f(N)=0$ for some $A$-submodule $N$ of $M$ with $\dim(M/N)<\infty$, i.e., $f\in M_r^{\circ}$;
\item[(b)] $f(NA)=0$ for some $A$-submodule $N$ of $M$ with $\dim(M/N)<\infty$;
\item[(c)] $f(ML)=0$ for some left ideal $L$ of $A$ with $\dim(A/L)<\infty$;
\item[(d)] $\dim (A\rhu f)<\infty$;
\item[(e)] $\dim (( -\rhu f)(M))<\infty$;
\item[(f)] $\rho_r(f)\in M^*\ot A^*$, where $\rho_r:M^*\ra ( M\ot A)^*$ is the dual map of the right action map $M\ot A\ra M$.
\end{enumerate}
\end{proposition}
\begin{proof}
Note that $(-\rhu f)(x)=f(x-)$ is an element in $ A^*$ for any $x\in M$.
Hence $(-\rhu f )(M)=f(M-)$ is a subspace of $A^*$.

For (a)$\Rightarrow$(b)$\Rightarrow$(d)$\Rightarrow$(a),
the proofs are similar to the proofs of (a)$\Rightarrow$(c)$\Rightarrow$(e)$\Rightarrow$(a) of
\prref{prop3.1}. For (d)$\Leftrightarrow$(f), the proof is similar to the proof of (e)$\Leftrightarrow$(g)
of \prref{prop3.1}.

(c)$\Leftrightarrow$(e)  Let $L$ be a left ideal of $A$ of finite codimension such that $f(ML)=0$.
Then $L^{\perp}:=\{g\in A^*\mid g(L)=0\}$ is finite dimensional.
Since $( -\rhu f)(x)(L)=(L\rhu f)(x)=f(xL)=0$ for any $x\in M$,
$( -\rhu f)(M)\ss L^{\perp}$, and hence  $( -\rhu f)(M)$ is finite dimensional.
Conversely, suppose $\dim (( -\rhu f)(M))<\infty$.
Let $L=((-\rhu f)(M))^{\perp}:=\{a\in A\mid (a\rhu f)(M)=f(Ma)=0\}$.
Then $L$ is a left ideal of $A$ with finite codimension and $f(ML)=0$.

(e)$\Leftrightarrow$(f) Suppose $\dim (( -\rhu f)(M))<\infty$.
Let $v_1,v_2,\cdots, v_n$ be a basis of $(-\rhu f)(M)$. Then for any $x\in M$,
$(-\rhu f)(x)=\sum_{i=1}^n l_i(x)v_i\in A^*$ for some $l_i(x)\in k$. Clearly, $l_i\in M^*$ for all $1\<i\<n$.
Moreover, for any $x\in M$ and $a\in A$, we have
$\rho_r(f)(x\ot a)=f(xa)=(a\rhu f)(x)=\sum_{i=1}^n l_i(x)v_i(a)=(\sum_{i=1}^n l_i\ot v_i)(x\ot a)$,
and so $\rho_r(f)=\sum_{i=1}^n l_i\ot v_i\in M^*\ot A^*$.
Conversely, suppose $\rho_r(f)=\sum_{i=1}^n l_i\ot v_i\in M^*\ot A^*$. Then
$(a\rhu f)(x)=f(xa)=\rho_r(f)(x\ot a)=\sum_{i=1}^nl_i(x)v_i(a)=(\sum_{i=1}^nl_i(x)v_i)(a)$
for all $a\in A$ and $x\in M$.
Hence $(-\rhu f)(x)=\sum_{i=1}^nl_i(x)v_i$ for all $x\in M$ and so $(-\rhu f)(M)$ is finite dimensional.
\end{proof}

\begin{proposition}\prlabel{prop3.5}
Let $M$ be a right $A$-module. Then $\rho_r(M_r^{\circ})\ss M_r^{\circ}\ot A^{\circ}$.
Moreover $(M_r^{\circ}, \rho_r)$ is a right $A^{\circ}$-comodule.
\end{proposition}
\begin{proof}
Let  $0\neq f\in M_r^{\circ}$. By \prref{prop3.4}, $\dim (A\rhu f)<\infty$, $\dim (( -\rhu f)(M))<\infty$
and $\rho_r(f)\in M^*\ot A^*$.
Hence we may write $\rho_r(f)=\sum_{i=1}^n f_i\ot g_i$ with $n$ minimal, where $f_i\in M^*$  and $g_i\in A^*$,
$1\<i\<n$. Then $\{f_1, f_2, \cdots, f_n\}$ and $\{g_1, g_2, \cdots, g_n\}$ are both linearly independent.
Thus, one may choose $x_1,x_2,\cdots, x_n\in M$ and $a_1, a_2, \cdots, a_n\in A$ such that $f_i(x_j)=g_i(a_j)=\delta_{ij}$, $1\<i, j\<n$.
In this case, $a_j\rhu f=\sum_{i=1}^ng_i(a_j)f_i=f_j$, and so $A\rhu f_j=A\rhu(a_j\rhu f)=Aa_j\rhu f\subseteq A\rhu f$,
$1\<j\<n$. Hence $\dim (A\rhu f_j)\<\dim (A\rhu f)<\infty$, and so $f_j\in M_r^{\circ}$ by \prref{prop3.4}, $1\<j\<n$.
Now for any $a\in A$ and $1\<j\<n$,
$f(x_ja)=\rho_r(f)(x_j\ot a)=\sum_{i=1}^n f_i(x_j)g_i(a)=g_j(a)$. Hence $g_j=f(x_j-)$ in $A^*$.
Thus, for any $a, b\in A$ and $1\<j\<n$, we have  $(g_j\lhu a)(b)=g_j(ab)=f(x_j(ab))=f((x_ja)b)$,
and so $g_j\lhu a=f((x_ja)-)=(-\rhu f)(x_ja)$. This shows that  $g_j\lhu A\ss (-\rhu f)(M)$, and hence
$g_j\lhu A$ is finite dimensional, $1\<j\<n$. By \prref{prop3.1},
$g_1, g_2, \cdots, g_n\in A^{\circ}$. It follows that $\rho_r(M_r^{\circ})\ss M_r^{\circ}\ot A^{\circ}$.

Finally, for any $f\in M_r^{\circ}$, we have $\rho_r(f)=\sum f_i\ot g_i$ for some $f_i\in M_r^{\circ}$ and
$g_i\in A^{\circ}$, and so
$$\begin{array}{c}
((1\ot \D)\rho_r(f))(x\ot a\ot b)=\sum_{i=1}^n f_i(x)\D(g_i)(a\ot b)=\sum_{i=1}^n f_i(x)g_i(ab)=f(x(ab)),\\
((\rho_r\ot 1)\rho_r(f))(x\ot a\ot b)=\sum_{i=1}^n \rho_r(f_i)(x\ot a)g_i(b)=\sum_{i=1}^n f_i(xa)g_i(b)=f((xa)b)
\end{array}$$
for any $x\in M$ and $a, b\in A$.
Hence $(1\ot \D)\rho_r=(\rho_r\ot 1)\rho_r$ since $x(ab)=(xa)b$.
This implies that $(M_r^{\circ}, \rho_r)$ is a right $A^{\circ}$-comodule.
\end{proof}

For left $A$-modules, we have similar results.

\begin{proposition}\prlabel{prop3.6}
Let $M$ be a left $A$-module and $f\in M^*$. Then the following are equivalent:
\begin{enumerate}
\item[(a)] $f(N)=0$ for some $A$-submodule $N$ of $M$ of finite codimension, i.e., $f\in M_l^{\circ}$;
\item[(b)] $f(AN)=0$ for some $A$-submodule $N\ss M$ of finite codimension;
\item[(c)] $f(RM)=0$ for some right ideal $R$ of $A$ of finite codimension;
\item[(d)] $\dim ( f\lhu A)<\infty$;
\item[(e)] $\dim ( f\lhu-)(M)<\infty$;
\item[(f)] $\rho_l(f)\in A^*\ot M^*$, where $\rho_l:M^*\ra ( A\ot M)^*$ is the dual map of the left action map
$A\ot M\ra M$.
\end{enumerate}
\end{proposition}
\begin{proof}
It is similar to the proof of \prref{prop3.4}.
\end{proof}

\begin{proposition}\prlabel{prop3.7}
Let $M$ be a left $A$-module. Then $\rho_l(M_l^{\circ})\ss A^{\circ}\ot M_l^{\circ}$.
Moreover, $(M_l^{\circ}, \rho_l)$ is a left $A^{\circ}$-comodule.
\end{proposition}
\begin{proof}
It is similar to the proof of \prref{prop3.5}.
\end{proof}

For an $A$-bimodule  $M$, define $M_{(A)}^{\circ}:=M_l^{\circ}\cap M_r^{\circ}$,
called {\it the finite dual of $A$-bimodule $M$}.
If there is no ambiguity, we write $M^{\circ}$ for $M_{(A)}^{\circ}$.

\begin{proposition}\prlabel{prop3.9}
Let $M$ be an $A$-bimodule. Then $\rho_l(M^{\circ})\ss A^{\circ}\ot M^{\circ}$ and
$\rho_r(M^{\circ})\ss M^{\circ}\ot A^{\circ}$.
Moreover, $(M^{\circ}, \rho_l, \rho_r)$ is an $A^{\circ}$-bicomodule.
\end{proposition}
\begin{proof}
Let $0\neq f\in M^{\circ}$. Then by the proof of \prref{prop3.5},
$\rho_r(f)=\sum_{i=1}^n f_i\ot g_i\in M_r^{\circ}\ot A^{\circ}$ and $f_i=a_i\rhu f$ for some $a_i\in A$.
Similarly, we have $\rho_l(f)=\sum_{i=1}^m l_i\ot h_i\in A^{\circ}\ot M_l^{\circ} $
and $h_i=f\lhu b_i$ for some $b_i\in A$.
Hence $a\rhu f=\sum_{i=1}^ng_i(a)f_i$ and $f\lhu a=\sum_{i=1}^m l_i(a)h_i$ for all $a\in A$.
For $a\in A$, $x\in M$ and $1\<i\<n$, $(f_i\lhu a)(x)= (a_i\rhu f)(ax)=f((ax)a_i)=f(axa_i)
=(f\lhu a)(xa_i)=\sum_{j=1}^m l_j(a)h_j(xa_i)=\sum_{j=1}^m l_j(a)(a_i\rhu h_j)(x)$,
and hence $f_i\lhu a=\sum_{j=1}^m l_j(a)(a_i\rhu h_j)$.
This shows that $f_i\lhu A\subseteq{\rm span}\{a_i\rhu h_j \mid 1\<j\<m\}$.
Hence ${\rm dim}(f_i\lhu A)<\infty$, and so  $f_i\in M_l^{\circ}$ by \prref{prop3.6}.
Thus, $f_i\in M^{\circ}$, which show that $\rho_r(f)=\sum_{i=1}^n f_i\ot g_i\in M^{\circ}\ot A^{\circ}$.
Similarly, one can show that $\rho_l(f)=\sum_{i=1}^m l_i\ot h_i\in A^{\circ}\ot M^{\circ}$.

Finally, it is easy to see that $(\rho_l\ot 1)\rho_r=(1\ot \rho_r)\rho_l$ since $a(xb)=(ax)b$ for any $a, b\in A$ and $x\in M$.
Therefore, $(M^{\circ}, \rho_l, \rho_r)$ is an $A^{\circ}$-bicomodule.
\end{proof}

\begin{remark}
For an algebra $A$, let $M= A$, the regular left and right module. Then $A_l^{\circ}=A_r^{\circ}=A^{\circ}$.
That is, the finite dual of the algebra $A$ coincides with the finite dual of the left (or right) regular $A$-module,
and the regular $A$-bimodule.
\end{remark}

Let $(A,I)$ be a Dorroh pair of algebras. Then $I$ is an algebra and an $A$-bimodule.
Let  $I^{\circ}$ denote the finite dual  of the algebra $I$, and $I_{(A)}^{\circ}$ denote the finite dual  of the $A$-bimodule $I$.
Then $I^{\circ}$ is a coalgebra and $I_{(A)}^{\circ}$ is an $A^{\circ}$-bicomodule.
Define $I^d:=I^{\circ}\cap I_{(A)}^{\circ}$.

\begin{lemma}\lelabel{lem3.11}
Let $(A,I)$ be a Dorroh pair of algebras. Then
\begin{enumerate}
\item[(a)] $I^d$ is a subcoalgebra of $I^{\circ}$.
\item[(b)] $I^d$ is an $A^{\circ}$-subbicomodule of $I^{\circ}_{(A)}$.
\item[(c)] The comodule actions $\rho_l$ and $\rho_r$ of $I^d$ are compatible with its comultiplication $\D$, i.e.,
$(\D\ot 1)\rho_r=(1\ot \rho_r)\D$, $(1\ot\D)\rho_l=(\rho_l\ot 1)\D$ and
$(\rho_r\ot 1)\D=(1\ot \rho_l)\D$.
 \end{enumerate}
\end{lemma}
\begin{proof}
Let $0\neq f\in I^d$. Then $f\in I^{\circ}$ and $f\in  I^{\circ}_{(A)}$. By \prref{prop3.3} and its proof,
we have $\D(f)=m^*(f)=\sum_{i=1}^n f_i\ot g_i\in I^{\circ}\ot I^{\circ}$, where
$\{f_1, f_2, \cdots, f_n\}$ is a basis of $I\rhu f$. Choose some elements
$x_1, x_2, \cdots, x_n\in I$ such that $x_i\rhu f= f_i$, and $y_1, y_2, \cdots, y_n\in I$ such that
$f_i(y_j)=\d_{ij}$. Then $f\lhu y_i=g_i$ and $g_i(x_j)=\d_{ij}$, $1\<i, j\<n$.
By \prref{prop3.9}, $\rho_l(f)\in A^{\circ}\ot I^{\circ}_{(A)}$ and
$\rho_r(f) \in I^{\circ}_{(A)}\ot A^{\circ}$.
Thus, we may assume $\rho_l(f)=\sum_{i=1}^m v_i\ot h_i\in A^{\circ}\ot I^{\circ}_{(A)}$ with $m$ minimal,
and $\rho_r(f)=\sum_{i=1}^s l_i\ot u_i \in I^{\circ}_{(A)}\ot A^{\circ}$ with $s$ minimal.
Then $\{h_1, h_2, \cdots, h_m\}$ is a basis of $f\lhu A$ and  $f\lhu a=\sum_{i=1}^m v_i(a)h_i$,
and $\{l_1, l_2, \cdots, l_s\}$ is a basis of $A\rhu f$ and $a\rhu f=\sum_{i=1}^su_i(a)l_i$, where $a\in A$.

(a) For $a\in A$, $x\in I$ and $1\<i\<n$,
we have $(a\rhu f_i)(x)=(x_i\rhu f)(xa)=f((xa)x_i)=f(x(ax_i))=\sum_{j=1}^n f_j(x)g_j(ax_i)$.
Hence $a\rhu f_i=\sum_{j=1}^n g_j(ax_i)f_j$, and so ${\rm dim}(A\rhu f_i)<\infty$.
Furthermore, we have $(f_i\lhu a)(x)=(x_i\rhu f)(ax)=f((ax)x_i)=f(a(xx_i))=\sum_{j=1}^m v_j(a)h_j(xx_i)
=\sum_{j=1}^m v_j(a)(x_i\rhu h_j)(x)$. Hence $f_i\lhu a=\sum_{j=1}^m v_j(a)(x_i\rhu h_j)$,
which implies ${\rm dim}(f_i\lhu A)<\infty$. By \prref{prop3.4} and \prref{prop3.6},
$f_i\in I^{\circ}_{(A)}$, and so $f_i\in I^d$. Similarly, one can show $g_i\in I^d$.
Thus, we have shown $\D(I^d)\ss I^d\ot I^d$. Hence $I^d$ is a subcoalgebra of $I^{\circ}$.

(b) Since $\{l_1, l_2, \cdots, l_s\}$ is a basis of $A\rhu f$, one may choose elements
$a_1,a_2,\cdots,a_s\in A$ such that $l_i=a_i\rhu f$.
Then for any $x, y\in I$, $(x\rhu l_i)(y)=(a_i\rhu f)(yx)=f((yx)a_i)=f(y(xa_i))=\sum_{j=1}^n f_j(y)g_j(xa_i)$,
and hence $x\rhu l_i=\sum_{j=1}^n g_j(xa_i)f_j$. Thus, ${\dim}(I\rhu l_i)<\infty$. Then by \prref{prop3.3}, $l_i\in I^{\circ}$,
and so  $l_i\in I^d$ since $l_i\in I^{\circ}_{(A)}$. If follows that $\rho_r(I^d)\ss I^d\ot A^{\circ}$.
Similarly, one can check that $\rho_l(I^d)\ss A^{\circ}\ot I^d$.
Therefore, $I^d$ is an $A^{\circ}$-subbicomodule of $I^{\circ}_{(A)}$.

(c) For any $x, y\in I$ and $a\in A$, we have
$((\D\ot 1)\rho_r(f))(x\ot y\ot a)=\sum_{i=1}^s \D(l_i)(x\ot y)u_i(a)=\sum_{i=1}^s l_i(xy)u_i(a)=f((xy)a)$
and $((1\ot \rho_r)\D(f))(x\ot y\ot a)=\sum_{i=1}^n f_i(x)\rho_r(g_i)(y\ot a)=\sum_{i=1}^n f_i(x)g_i(ya)=f(x(ya))$.
Hence $(\D\ot 1)\rho_r=(1\ot \rho_r)\D$ since $(xy)a=x(ya)$.
Similarly, one can check $(1\ot\D)\rho_l=(\rho_l\ot 1)\D$ and
$(\rho_r\ot 1)\D=(1\ot \rho_l)\D.$
\end{proof}

Let $C$ be a coalgebra. Then $C^*$ is an algebra with the multiplication given by
$(fg)(c)=\sum f(c_1)g(c_2)$ for any $f, g\in C^*$ and $c\in C$.
Let $V$ be a left (resp. right) $C$-comodule.
Then $V^*$ is a left (resp. right) $C^*$-module
with the action given by $(f\eta)(v)=\sum f(v_{(-1)})\eta(v_{(0)})$ 
(resp. $(\eta f)(v)=\sum\eta(v_{(0)}) f(v_{(1)})$)
for any $f\in C^*$, $\eta\in V^*$ and $v\in V$.

\begin{theorem}
\begin{enumerate}
\item[(a)] Let $(A, I)$ be a Dorroh pair of algebras. Then $(A^{\circ}, I^d)$ is a Dorroh pair of coalgebras
and $A^{\circ}\ltimes_d I^d\cong (A\ltimes_d I)^{\circ}$ as coalgebras.
\item[(b)] Let $(C, P)$ be a Dorroh pair of coalgebras. Then $(C^*, P^*)$ is a Dorroh pair of algebras
and $C^*\ltimes_d P^*\cong (C\ltimes_d P)^*$ as algebras.
\end{enumerate}
\end{theorem}

\begin{proof}
(a) By \prref{prop3.3} and \leref{lem3.11}, $(A^{\circ},I^d)$ is a Dorroh pair
of coalgebras. Hence one can form a coalgebra Dorroh extension $A^{\circ}\ltimes_d I^d$.

One may regard $A^*$ and $I^*$ as subspaces of $(A\ltimes_d I)^*$:
for any $f\in A^*$ and $g\in I^*$, define $f, g\in (A\ltimes_d I)^*$ by
$f(a, x)=f(a)$ and $g(a, x)=g(x)$, $a\in A$, $x\in I$.
In this case, we have $(A\ltimes_d I)^*=A^*\oplus I^*$ as vector spaces.
Let $\pi_A: (A\ltimes_d I)^*\ra A^*$, $\phi\mapsto\phi_A$ and
$\pi_I: (A\ltimes_d I)^*\ra I^*$, $\phi\mapsto\phi_I$ be the corresponding projections.
Then an element $\phi \in (A\ltimes_d I)^*$ can be uniquely expressed as $\phi=\phi_A+\phi_I$,
where $\phi_A(a,x)=\phi_A(a)=\phi(a,0)$ and $\phi_I(a, x)=\phi_I(x)=\phi(0, x)$, $(a, x)\in A\ltimes_d I$.

Now let $\phi=\phi_A+\phi_I\in (A\ltimes_d I)^*$, Then for any $(a,x), (b,y)\in A\ltimes_d I$, we have
$((a,x)\rhu \phi)(b,y)=\phi((b,y)(a,x))=\phi(ba,bx+ya+yx)
=\phi_A(ba)+\phi_I(bx)+\phi_I(ya)+\phi_I(yx)
=(a\rhu \phi_A)(b)+(\phi_I\lhu -)(x)(b)+(a\rhu \phi_I)(y)+(x\rhu \phi_I)(y)$.
Hence we have
\begin{equation}\eqlabel{e16}
(a,x)\rhu \phi=(a\rhu \phi_A)+(\phi_I\lhu -)(x)+(a\rhu \phi_I)+(x\rhu \phi_I),
\end{equation}
and so $(A\ltimes_d I)\rhu \phi \ss (A\rhu \phi_A) + (\phi_I\lhu -)(I)+(A\rhu \phi_I)+
(I\rhu \phi_I)$. If $\phi_A\in A^{\circ}$ and $\phi_I\in I^d$,
then ${\rm dim}(A\rhu \phi_A)<\infty$, ${\rm dim}((\phi_I\lhu -)(I))<\infty$, ${\rm dim}(A\rhu \phi_I)<\infty$
and ${\rm dim}(I\rhu \phi_I)<\infty$ by \prref{prop3.1}, \prref{prop3.4} and \prref{prop3.6}.
Hence ${\rm dim}((A\ltimes_d I)\rhu \phi)<\infty$, and so $\phi\in(A\ltimes_d I)^{\circ}$.
Conversely, if $\phi\in(A\ltimes_d I)^{\circ}$ then ${\rm dim}((A\ltimes_d I)\rhu \phi)<\infty$ by \prref{prop3.1}.
From Eq.\equref{e16}, one gets that $a\rhu \phi_A=\pi_A((a,0)\rhu\phi)$,
$(\phi_I\lhu -)(x)=\pi_A((0,x)\rhu\phi)$, $a\rhu \phi_I=\pi_I((a,0)\rhu\phi)$ and $x\rhu \phi_I=\pi_I((0,x)\rhu\phi)$.
Hence $A\rhu \phi_A\subseteq\pi_A((A\ltimes_d I)\rhu \phi)$, $(\phi_I\lhu -)(I)\subseteq\pi_A((A\ltimes_d I)\rhu \phi)$,
$A\rhu \phi_I\subseteq\pi_I((A\ltimes_d I)\rhu \phi)$ and $I\rhu \phi_I\subseteq\pi_I((A\ltimes_d I)\rhu \phi)$.
This implies that $A\rhu \phi_A$, $(\phi_I\lhu -)(I)$, $A\rhu \phi_I$ and $I\rhu \phi_I$ are all finite dimensional,
and so $\phi_A\in A^{\circ}$ and $\phi_I\in I^d$ by \prref{prop3.1}, \prref{prop3.4} and \prref{prop3.6}.
Thus, $(A\ltimes_d I)^{\circ}=A^{\circ}\oplus I^d$ as vector spaces

Define a map $F:(A\ltimes_d I)^{\circ}\ra A^{\circ}\ltimes_d I^d$
by $F(\phi)=(\phi_A, \phi_I)$. Then by the discussion above, $F$ is a linear isomorphism.
For any $\phi\in(A\ltimes_d I)^{\circ}$, $a, b\in A$ and $x, y\in I$, we have
$$((F\ot F)\D (\phi))((a,x)\ot (b,y))=\phi((a,x)(b,y))=\phi(ab,ay+xb+xy)$$
and
$$\begin{array}{rl}
&(\D(F(\phi)))((a,x)\ot (b,y))\\
=&(\sum((\phi_{A})_1,0)\ot((\phi_A)_2,0)+\sum ((\phi_I)_{(-1)},0)\ot (0,(\phi_I)_{(0)})\\
&+\sum (0,(\phi_I)_{(0)})\ot ((\phi_I)_{(1)},0)+\sum (0,(\phi_I)_1)\ot (0,(\phi_I)_2))((a,x)\ot (b,y))\\
=&\phi_A(ab)+\phi_I(ay)+\phi_I(xb)+\phi_I(xy)=\phi(ab,ay+xb+xy).\\
\end{array}$$
This shows that $F$ is a coalgebra homomorphism. Therefore,
$(A\ltimes_d I)^{\circ}\cong A^{\circ}\ltimes_d I^d$ as coalgebras.

(b) At first, it is straightforward to check that $(C^*, P^*)$ is a Dorroh pair of algebras,
and hence one can form an algebra Dorroh extension $C^*\ltimes_d P^*$. 
Then similarly to (a), there is a canonical linear isomorphism
$F: C^*\ltimes_d P^*\cong (C\ltimes_d P)^*$ given by $(F(f,g))(c,p)=f(c)+g(p)$,
$f\in C^*$, $g\in P^*$, $c\in C$ and $p\in P$.
A straightforward verification shows that $F$ is an algebra homomorphism.
\end{proof}

\begin{remark}
If $A$ is a finite dimensional algebra, then $A^*=A^{\circ}$. Furthermore, if $(A, I)$ is a Dorroh pair of algebras
with ${\rm dim}(A)<\infty$ and ${\rm dim}(I)<\infty$, then $(A\ltimes_d I)^*=A^*\ltimes_d I^*$ as coalgebras.
\end{remark}

At the end of this section, we consider the finite duals of modules over unital algebras.
The finite dual of a unital algebra $A$ was well described before, see \cite[Lemma 9.1.1]{Mo}.
For a module over a unital algebra, we have the following propositions.

\begin{proposition}\prlabel{prop3.14}
Let $A$ be a unital algebra and $M$ a right $A$-module.
Let $f\in M^*$. Then the following are equivalent:
\begin{enumerate}
\item[(a)] $f\in M_r^{\circ}$;
\item[(b)] $f(MI)=0$ for some ideal $R$ of $A$ of finite codimension;
\item[(c)] $\dim f(M-A)<\infty$, where $f(M-A)=A\rhu((-\rhu f)(M))=(-\rhu (A\rhu f))(M)=((-A)\rhu f)(M)$.
\end{enumerate}
\end{proposition}

\begin{proof}
For any $a\in A$ and $x\in M$, $f(x-a)\in A^*$ with $(f(x-a))(b)=f(xba)$, $b\in A$.
Hence $f(M-A)$ is a subspace of $A^*$.

(a)$\Rightarrow$(b) Assume $f\in M_r^{\circ}$. Then by \prref{prop3.4},
there is a left ideal $L$ of $A$ with $\dim A/L<\infty $ such that $f(ML)=0$.
Hence $A/L$ is a left $A$-module and there is a corresponding algebra homomorphism $\phi: A\ra \End (A/L)$.
Let $I:=\Ker(\phi)=\{a\in A\mid aA\subseteq L\}$. Then $I$ is an ideal of $A$ of finite codimension.
Since $A$ is unital, $I\subseteq L$. Hence $f(MI)=0$.

(b)$\Rightarrow$(c) Suppose $I$ is an ideal of $A$ of finite codimension with $f(MI)=0$.
Since $MIA\ss MI$, $f(MIA)=0$. Hence $f(M-A)\ss I^{\perp}:=\{g\in A^*| g(I)=0\}$,
and so $\dim f(M-A)<\infty$.

(c)$\Rightarrow$(a) Suppose $\dim f(M-A)<\infty$. Since $A$ is unital, $(-\rhu f)(M)=f(M-)\ss f(M-A)$.
Hence $\dim((-\rhu f)(M))<\infty$, and so $f\in M_r^{\circ}$ by \prref{prop3.4}.
\end{proof}

\begin{proposition}\prlabel{prop3.15}
Let $A$ be a unital algebra and $M$ a left $A$-module. Let $f\in M^*$. Then the following are equivalent:
\begin{enumerate}
\item[(a)] $f\in M_l^{\circ}$;
\item[(b)] $f(IM)=0$ for some ideal $I$ of $A$ of finite codimension;
\item[(c)] $\dim f(A-M)<\infty$, where $f(A-M)=((f\lhu -)(M))\lhu A=((f\lhu A)\lhu-)(M)=(f\lhu (A-))(M)$.
\end{enumerate}
\end{proposition}
\begin{proof}
It is similar to the proof of \prref{prop3.14}.
\end{proof}

\centerline{ACKNOWLEDGMENTS}

This work is supported by NNSF of China (No. 11571298, 11971418) and
Graduate student scientific research innovation projects in Jiangsu Province, No. XKYCX18\_036

\end{document}